\documentclass[12pt,reqno]{amsart}
\usepackage{amssymb,amsmath,amsthm}
\usepackage[lite,bibtex-style,backrefs]{amsrefs}
\usepackage{graphicx}
\usepackage{mathtools}
\usepackage{tikz,tkz-euclide}
\usetkzobj{all}
 
\usepackage{amssymb,amsmath,amsthm}
\usepackage[all]{xy}
\usepackage{xcolor,colortbl}
\usepackage{soul}

\newtheorem{theorem}{Theorem}[section]
\newtheorem{proposition}[theorem]{Proposition}
\newtheorem{lemma}[theorem]{Lemma}
\newtheorem{corollary}[theorem]{Corollary}
 
\theoremstyle{definition}
\newtheorem{definition}[theorem]{Definition}

\newtheorem{remark}[theorem]{Remark}

\numberwithin{equation}{section}
\newcommand{\sfA}{\mathsf{A}}

\newcommand{\bbZ}{{\mathbb{Z}}}
\newcommand{\bbP}{{\mathbb{P}}}
\newcommand{\bbG}{{\mathbb{G}}}

\newcommand{\bbQ}{{\mathbb{Q}}}
\newcommand{\bbF}{{\mathbb{F}}}

\newcommand{\bbR}{\mathbb{R}}

\newcommand{\bfF}{{\mathbf{F}}}

\newcommand{\SL}{\operatorname{SL}}

\newcommand{\Aut}{\operatorname{Aut}}

\newcommand{\Pic}{\operatorname{Pic}}

\newcommand{\cha}{\operatorname{char}}

\newcommand{\mult}{\operatorname{mult}}
\newcommand{\da}{\dasharrow}
\newcommand{\bsm}{\left(\begin{smallmatrix}}
\newcommand{\esm}{\end{smallmatrix}\right)}
\newcommand{\la}{\langle}
\newcommand{\ra}{\rangle}
\newcommand{\res}{\operatorname{res}}

\newcommand{\Exc}{\operatorname{Exc}}

\newcommand{\half}{\frac{1}{2}}

\newcommand{\calO}{\mathcal{O}}

\newcommand{\calC}{\mathcal{C}}
\newcommand{\calD}{\mathcal{D}}

\newcommand{\calE}{\mathcal{E}}

\newcommand{\frakS}{\mathfrak{S}}
\newcommand{\frakA}{\mathfrak{A}}

\newcommand{\frakf}{\mathfrak{f}}

\newcommand{\beq}{\begin{equation}}
\newcommand{\eeq}{\end{equation}}

\definecolor{green}{rgb}{0,0.5,0}

\date{}

\title{Chilean configuration of conics, lines and points}

\author[I. Dolgachev]{Igor Dolgachev}

\address{Department of Mathematics, University of Michigan, Ann Arbor, USA}

\email{idolga@umich.edu}

\author[A. Laface]{Antonio Laface}

\address{Departamento de Matem\'atica, Facultad de Ciencias F\'isicas y Matem\'aticas, Universidad de Concepci\'on, Concepci\'on, Chile}

\email{alaface@udec.cl}

\author[U. Persson]{Ulf Persson}

\address{University of Gothenburg Chalmers, Mathematical Sciences, Algebra and geometry, Gothenburg, Sweden}

\email{ulf.persson@gu.se} 

\author[G. Urz{\'u}a]{Giancarlo Urz{\'u}a}

\address{Facultad de Matem{\'a}ticas, Pontificia Universidad Cat{\'o}lica de Chile, Santiago, Chile}

\email{urzua@mat.uc.cl}

\begin{document}

\maketitle

\begin{abstract}

Using the theory of rational elliptic fibrations, we construct and discuss a one parameter family of configurations of $12$ conics and $9$ points in the projective plane that realizes an abstract configuration $(12_6,9_8)$. This is analogous to the famous Hesse configuration of $12$ lines and $9$ points forming an abstract configuration $(12_3,9_4)$. We also show that any Halphen elliptic fibration of index $2$ with four triangular singular fibers arises from such configuration of conics.
\end{abstract}

\section{Introduction}

The famous Hesse configuration $(12_3,9_4)$ of $12$ lines and $9$ points, with $3$ points on each line and $4$ lines through each point, can be geometrically realized by the Hesse pencil of cubic curves 
$$x^3+y^3+z^3+txyz = 0$$
with $9$ base points and $4$ reducible members each consisting of three lines. Each line contains three base points, and each base point is on four reducible members. By blowing up the base points, we obtain a rational elliptic surface with four reducible fibers of type $I_3$ in Kodaira's notation for singular fibers of elliptic fibrations on algebraic surfaces. A pencil of cubic curves is the first in the series of Halphen pencils of plane curves whose general member is of degree $3m$ with nine $m$-multiple points (including infinitely near points). The number $m$ is called the index of the Halphen pencil. 

In this paper we give an explicit construction and a complete discussion of an Halphen pencil of index $2$ that contains exactly four reducible members, each is the union of three smooth conics. Each of these $12$ conics contains $6$ base points and each base point lies on $8$ of these conics, forming an abstract configuration $(12_6,9_8)$ which is analogous to the Hesse configuration. Even more surprising is that one can find the dual Hesse configuration $(9_4,12_3)$ of $9$ lines and $12$ points embedded in our configuration of conics. It consists of the $12$ singular points of the reducible members of the Halphen pencil and $12$ lines that are analogous to the $9$ harmonic polar lines in the Hesse configuration. 

We work over any algebraically closed field of characteristic different from $2$ and $3$.

The paper originates from a question of Piotr Pokora who asked one of the authors about an interesting configuration of conics in the plane. The present construction arose from discussions during the fourth Latin American School on Algebraic Geometry and its Applications (ELGA IV) held on December 2019 in Talca, Chile. This explains the name for our configuration of conics. 

The paper is organized as follows.
In Section 2 we discuss Halphen
pencils and in particular provide 
structure theorems for $(-1)$-curves
on rational elliptic surfaces coming
from such pencils.
Starting from this point we focus on 
rational elliptic surfaces whose jacobian 
surface is the Hesse surface.
The definition of the latter surface 
is given in Section 3.
Section 4 is devoted to the explicit 
construction of Chilean surfaces: a 
one-parameter family of rational elliptic
surfaces of index two whose jacobian 
is the Hesse surface.
In Section 5 we show that any such
rational elliptic surface is a Chilean surface.
Section 6 introduces a double plane model
for elliptic surfaces of index two which allows 
one to describe their $(-1)$-curves.
Such description is given in the 
same section for the Chilean surfaces.
Section 7 deals with rational elliptic 
surfaces of higher index whose jacobian 
surface is the Hesse surface.
Finally in Section 8 we discuss some few properties and applications such as log Chern numbers for configurations of lines and conics which are naturally associated to our Halphen pencils.


\subsection*{Acknowledgements}
Antonio Laface was supported by the FONDECYT regular grant 1190777. Part of this work was done during a visit of Ulf Persson to the Department of Mathematics at the Universidad Cat\'olica del Norte from August 2019 to January 2020, which was funded by the CONICYT project MEC80180007 and the
Vice-Rector\'\i a de Investigaci\'on y Desarrollo Tecnol\'ogico
of the university.
Giancarlo Urz{\'u}a was supported by the FONDECYT regular grant 1190066.

\section{Halphen pencils}

Let $\lambda F+\mu G = 0$ be any pencil of plane curves whose general fiber is birationally isomorphic to an elliptic curve, a smooth  projective curve of genus $1$. It defines a rational map 
\beq\label{diag1}
\phi \colon \bbP^2\da\bbP^1, \quad (x_0:x_1:x_2)\mapsto (F(x_0,x_1,x_2):G(x_0,x_1,x_2)).
\eeq 
Let $\pi \colon X\to \bbP^2$ be a resolution of base points of the pencil. We have the following commutative diagram 
$$\xymatrix{&X\ar[dl]_{\pi}\ar[dr]^{f'}& \\
\bbP^2\ar@{-->}[rr]^\phi&&\bbP^1},$$ where $f'$ is an elliptic fibration. The birational morphism $\pi$ admits a factorization
\beq
\label{factorization}
\pi \colon X = X_N\overset{\pi_N}{\longrightarrow} X_{N-1}\overset{\pi_{N-1}}{\longrightarrow}\cdots \overset{\pi_2}{\longrightarrow} X_1\overset{\pi_1}{\longrightarrow} X_0 =\bbP^2,
\eeq
where each morphism $\pi_i \colon X_i\to X_{i-1}$ is the blow-up of one point $x_i\in X_{i-1}$. For any $i \geq j$, let $\pi_{i,j}:= \pi_j\circ\cdots \circ\pi_i \colon X_i\to X_{j-1}$. A point $x_i\in X_i$ with $\pi(x_i) = x_j\in X_{j-1}$ is called \emph{infinitely near} to $x_j$ (of order $i-j$). The points $\pi_{i,1}(x_i)\in \bbP^2, i = 1,\ldots,N,$ are the intersection points of two general members of the pencil. Their number could be less than $N$. 

The points $\pi_{i,2}(x_i)\in X_{1}$, $i=2,\ldots, N$, are the intersection points of the proper transforms of two general members of the pencil in $X_1$, and so on, until the proper transforms of the members of the pencil on $X$ has no intersection points and hence we obtain a morphism $f'\colon X\to \bbP^1$ whose general fiber is birationally isomorphic to the general member of the pencil. Since $X$ is smooth, a general fiber $F$ of $f$ being birationally isomorphic to an elliptic curve must be smooth, and hence $F$ is an elliptic curve. Let $\sigma \colon X\to S$ be a birational morphism to a relatively minimal model of the fibration $f' \colon X\to \bbP^1$. By definition,  $f'$ is a composition of $\sigma$ with a morphism $f \colon S\to \bbP^1$ whose fibers do not contain $(-1)$-curves, smooth rational curves with self-intersection $-1$. Let $F$ be a general fiber of $f$. Let $D$ be a divisor on $S$ with $\calO_S(D)\cong \calO_S(K_S)$. Restricting $D$ to the general fiber $S_\eta$ we obtain a divisor linearly equivalent to zero on $S_{\eta}$. Replacing it by a linearly equivalent divisor we obtain that the restriction of $D$ to $S_\eta$ is the zero divisor, i.e.   the support of $D$ is contained in fibers of $f$.

We use the well-known fact that the restriction of the intersection form of divisor classes on $S$ to the subgroup generated by irreducible components of a fiber is semi-negative definite  with the radical generated over $\bbQ$ by $F$ itself \cite[Proposition VIII.3]{Beauville}. This implies that  any irreducible component $\Theta$ of a reducible fiber has negative self-intersection,  hence we obtain 
$\Theta^2\le -2$, and, by adjunction formula, $\Theta\cdot K_S\ge 0$. Since $\Theta$ is a part of $F$, we obtain $\Theta\cdot K_S = 0$. Since $K_S$ has a representative contained in fibers, we obtain  that 
\beq\label{canformula2}
F= -mK_S
\eeq 
for some rational number $m$. Intersecting both sides with the divisor class of a $(-1)$-curve $E$ on $S$, we get that $m = F\cdot E$ is a positive integer. Let $\pi':S\to V$ be a birational morphism to a minimal rational surface $V$. It follows from \eqref{canformula2} that $S$ has no smooth rational curves with self-intersection less than $-2$.  This easily implies that we can choose $V$ to be $\bbP^2$. Let us take  $X$ in \eqref{factorization} to be our surface $S$. 

Let $\calE_i = \pi_{N,i}^*(\pi_i^{-1}(x_i)), i = 1,\ldots,9$. Then $e_0 := \pi^*(\textrm{line}), e_i := \calE_i$ form a basis in the Picard group $\Pic(X)$ of divisor classes on $X$ satisfying $e_0^2 = 1, e_i^2 = -1, e_i\cdot e_j = 0, i\ne j$. It is called a \emph{geometric basis}. The known behavior of the canonical class under the blow-up of a point gives
\beq\label{canformula1}
K_S = -3e_0+e_1+\cdots+e_N.
\eeq
Since $K_S^2 = 0$,  formula \eqref{canformula2} shows that $N = 9$. Thus we obtain a commutative diagram of rational maps
\[
\xymatrix{&X\ar[dl]_{f'}\ar[r]\ar[d]^{\pi}\ar[r]^\sigma&S\ar[d]^{\pi'}\ar[dr]^f&\\
\bbP^1&\bbP^2\ar@{-->}[l]\ar@{-->}[r]^T&\bbP^2\ar@{-->}[r]&\bbP^1}
\]
The birational map $T$ transforms our original pencil to a pencil of elliptic curves with $9$ base points $y_1,\ldots,y_9$ with some of them may be infinitely near points. The existence of such transformation $T$ was first proven by E. Bertini in 1877 \cite{Bertini} (a modern proof following the arguments from above was first given in \cite{Dolgachev0}). Its proper transform on $S$ belongs to the linear system $|-mK_S|$. Applying formulas \eqref{canformula1} and \eqref{canformula2}, we obtain that a general member $F$ of the elliptic pencil on $S$ satisfies $F\sim mF_0$. Thus its image in the plane is a curve $F_{3m} = 0$ of degree $3m$ with $m$-multiple points at $y_1,\ldots,y_9$. The pencil can be written in the form 
\beq\label{halphen}
\lambda F_{3m}+\mu G_3^m = 0.
\eeq
It is called an \emph{Halphen pencil} (of \emph{index} $m$) in honor of G. Halphen who was the first to give a detailed discussion of the properties of such pencils \cite{Halphen}. Of course, if $m = 1$, this is a pencil of plane cubics. Note that, if $m > 1$, one of the fibers of $f$ is of the form $F = mF_0$ (because $F\sim -mK_S$). It is called  the $m$-multiple fiber of the elliptic fibration. Obviously there is only one multiple fiber since otherwise we can take $F_{3m}$ to be equal to $F_{3s}^{m/s}$ for some $s$ dividing $m$  and obtain that a general member of the pencil is reducible. 

\begin{lemma}  Let $f \colon S\to \bbP^1$ be a relatively minimal elliptic fibration on a rational surface $S$ defined by an Halphen pencil of index $m$ with a fiber $F = mF_0$. Let $(e_0,e_1,\ldots,e_9)$ be the geometric basis on $S$ defined by a birational morphism $\pi:S\to \bbP^2$. Then $\calO_{F_0}(-K_S) = \calO_{F_0}(3e_0-e_1-\cdots-e_9)$ has order $m$ in the Picard group $\Pic(F_0)$. Moreover, if $mF_0$ is a multiple fiber with $m$ prime to the characteristic $p =\cha(\Bbbk)$ of $\Bbbk$ (resp. divisible by $p$), then $F_0$ is of Kodaira's type $I_n$ (resp. of other types and $m = p$). 
\end{lemma}

\begin{proof} We have already used that $\calO_F(K_S) \cong \calO_F$ for a general fiber $F$ of the fibration. Since we can find a fiber $F$ disjoint from $F_0$, we obtain 
$$\calO_{F_0}(F) \cong \calO_{F_0}(-mK_S) \cong \calO_{F_0}(-K_S)^{\otimes m}\cong \calO_{F_0}.$$
This shows that the order of $\calO_{F_0}(-K_S)$ in $\Pic(F_0)$ divides $m$. Suppose it is equal to $k < m$. Since $|lF_0|+(m-l)F_0 \subset |mF_0|$ for any $l< m$ and $|mF_0|$ is an irreducible pencil, we have $h^0(lF_0) = 1$. Applying Riemann-Roch, we deduce from this that $h^1(lF_0) = 0$. Then the exact sequence 
$$0\to \calO_{S}((k-1)F_0)\to \calO_S(kF_0) \to \calO_{F_0}(kF_0) \cong \calO_{F_0} \to 0$$
shows that there exists a section $s$ of $\calO_S(kF_0)$ whose divisor of zeros $D$ is an element of the linear system $|kF_0|$ that is disjoint from $kF_0$. This shows that $kF_0$ moves in a pencil generated by $D$ and $kF_0$, contradicting equality $h^0(kF_0) = 1$ from above. 

The last assertion follows from the fact that the $m$-torsion subgroup $\Pic(F_0)$ is non-trivial only in the cases from the assertion from the lemma \cite[Chapter 4, \S 1]{CDL}. 
\end{proof}

\begin{remark} We used here that $S$ is a rational elliptic surface. In general, an elliptic surface may have multiple fibers $mF_0$ such that $\calO_{F_0}(F_0)$ is of order strictly dividing $m$ and even could be equal  to 1. Such fibers are called wild and  occur only if $\cha(\Bbbk)$ divides $m$.  Any elliptic surface $X$ with $H^1(X,\calO_X) = 0$, as in our case, has no wild multiple fibers.
\end{remark}

Assume that $F_0$ is a smooth elliptic curve, and let $\{G_3 = 0\}$ be its image in the plane. Then one can interpret the assertion of the previous lemma by saying that choosing a group law on the plane cubic $G_3 = 0$, the base points $y_1,\ldots,y_9$ add up to a $m$-torsion point. This follows from the linear equivalence $m(y_1+\cdots+y_9)\sim 3mh$, where $h$ is the divisor class of the intersection of the cubic with a general line in the plane. This makes sense even when some of the points are infinitely near. 

Conversely, choose $9$ distinct  points $y_1,\ldots,y_9$ on a smooth plane cubic $C$ such that 
the divisor class of $y_1+\cdots+y_9-3h$ is of order $m$ in the Picard group. Choose an inflection point $q$ with the  tangent line $\{L = 0\}$ to be the zero point in the group law on $C: \{G_3 = 0\}$. Then the points $y_i$ taken with multiplicity $m$ add up to zero. Let $F_{3m}$ be a homogeneous polynomial such that the restriction of the rational function $F_{3m}/L^{3m}$ to $C$  is a rational function $\phi$ with $\text{div}(\phi) = m(y_1+\cdots+y_9)-3mq$. The choice of $F_{3m}$ is not unique and one can choose $P_{3m}$ such that the points $y_1,\ldots,y_9$ are points on the curve $F_{3m} = 0$ of multiplicities $m$ \cite[Lemma 4.4]{DolgachevInvariants}. The pencil $\{\lambda F_{3m}+\mu G_3^m = 0 \}$ is an Halphen pencil of index $m$.

Let $F_\eta$ be the generic fiber of an elliptic fibration considered as an elliptic curve over the field $K = \Bbbk(t)$ of rational functions on $\bbP^1$. A rational point of $F_\eta(L)$ over a finite extension $L/K$ defines, by passing to its Zariski closure on $S$, an irreducible curve such that the restriction of $f$ to it is a finite cover of degree equal to $d= [L:K]$. We call such a curve a $d$-section. In particular, a rational point in $F_\eta(K)$ defines a section of the fibration, and conversely any section arises in this way from a rational point on $F_\eta$.

Suppose $f \colon S \to \bbP^1$ is a rational elliptic surface arising from an Halphen pencil of index $m$. Since one of the fibers is of the form $F= mF_0$, intersecting a $d$-section with $F$, we obtain that $m$ divides $d$. On other hand, any $(-1)$-curve $E$ on $S$ satisfies $E\cdot F = -mE\cdot K_S = m$, hence it defines a $m$-section.

Suppose $m > 1$, then $F_\eta(K) =\emptyset$. The Jacobian variety $\textrm{Jac}(F_\eta)$ parametrizes the divisor classes of degree $0$ on $F_\eta$. It is an elliptic curve with the group law defined by the addition of the divisor classes.  It becomes isomorphic to $F_\eta$ over a field extension $L$ of $K$ such that $F_\eta(L)\ne \emptyset$. The curve $F_\eta$ is a torsor (= principal homogenous space) over $\textrm{Jac}(F_\eta)$ . This means that there is an action morphism $a \colon \textrm{Jac}(F_\eta)\times F_\eta\to F_\eta$ over $K$ such that the morphism $(a,p_2) \colon \textrm{Jac}(F_\eta)\times F_\eta\to F_\eta \times F_\eta$, where $p_2$ is the second projection map, is an isomorphism. The theory of minimal models of algebraic surfaces allows one to find a relatively minimal elliptic fibration $j \colon J\to \bbP^1$ with the generic fiber $J_\eta$ isomorphic to $\textrm{Jac}(F_\eta)$. It is called the \emph{jacobian fibration} of $f$. One can show that the isomorphism class of a non-trivial torsor $F_\eta$ of $J_\eta$ realized as the generic fiber of a rational elliptic surface is uniquely determined by a choice of the data $(x_0,\tau)$ that consists of a point $x_0\in \bbP^1$ and a $m$-torsion class $\tau$ in the connected component of the identity of the Picard scheme  of the fiber $J_{x_0}$ of $j$ over the point $x_0$ \cite[Chapter 4,\S 8]{CDL}. The corresponding fiber of $f \colon S\to \bbP^1$ over $x_0$ is the $m$-multiple fiber $mF_0$. All other fibers of $f$ are isomorphic to fibers of $j$. We refer to Remark \ref{moduli} where this construction is made very explicit for Halphen surfaces of index $2$.
   
Thus we obtain that to any Halphen pencil of index $m > 1$ \eqref{halphen} one can associate an Halphen pencil of index $1$ where one of the members is isogenous of degree $m$ to the cubic curve $\{G_3 = 0\}$. All other members are birationally isomorphic to the members of the Halphen pencil.

\begin{remark} Over the complex numbers Kodaira introduced a highly transcendental transformation (a so called logarithmic transformation, see \cite{BPV}) that for any (positive) integer $m$ makes any smooth fibre (or more generally any so called semi-stable fiber $I_n$) into a fiber with multiplicity $m$ while leaving the complements biholomorphic (but not birational in general). In particular, other fibers are left unscathed, however it causes havoc among transversal divisors. As the transform has an inverse, this is the 'only' way multiple fibers occur. There is also the inverse transformation that coincides, in the category of algebraic surfaces, with taking the jacobian fibration. One can prove, using the formula for the canonical class of an elliptic surfaces that Halphen surface are the only ones that are remain of the same birational type as their jacobian surface. This is a unique situation, which accounts for its particular allure.
\end{remark} 

We begin with the following well-known lemma of which we provide a proof for the sake of completeness.

\begin{lemma}
\label{lem:-1}
Let $S$ be a smooth rational surface
with nef anticanonical divisor $-K_S$.
Then any divisor $E$ such that 
$E^2 = E\cdot K_S = -1$ which has
non-negative intersection with all
the $(-2)$-curves of $S$ is linearly
equivalent to a $(-1)$-curve.
\end{lemma}

\begin{proof}
Assume $E$ satisfies the hypothesis of the 
theorem.
By Riemann-Roch $E$ is linearly equivalent
to an effective divisor. Without loss of
generality, we can assume $E$ itself to be 
effective.
Since $-K_S$ is nef it has non-negative
intersection with any irreducible curve
and, by adjunction, the only curves which
have intersection $0$ with $-K_S$ are
the $(-2)$-curves.
It follows that $E = E' + R$, where
$E'$ is a $(-1)$-curve and $R$ is a 
non-negative sum of $(-2)$-curves.
From $E^2=-1$ we deduce $E'\cdot R+
(E'+R)\cdot R = 0$.
Since $E'\cdot R\geq 0$, being 
$E'$ irreducible and not contained in $R$,
we deduce $0\leq E\cdot R = (E'+R)\cdot R \leq 0$. Thus $E'\cdot R = R^2 = 0$.
Since $R$ is in $K_S^\perp$, which is a 
negative-semidefinite lattice, we deduce
that $R\sim nK_S$ for some $n\in\mathbb Z$.
The equation $(E'+ nK_S)^2 = -1$ forces
$n=0$, so that $E = E'$.
\end{proof}

\begin{proposition}
\label{pro:res}
Let $S$ be a rational elliptic surface which is relatively minimal, that is, there are no 
$(-1)$-curves in the fibers of the elliptic
fibration $S\to\mathbb P^1$.
Assume $F_0\in |-K_S|$ is smooth. 
Then the kernel of the restriction map 
\[
 \res\colon \Pic(S)\to\Pic(F_0)
\]
contains the subgroup $\Lambda$ generated 
by the classes of the $(-2)$-curves of $S$.
Moreover if $\pi$ is extremal, that is 
$\Lambda$ has rank $9$, then $\Lambda
= \ker(\res)$.
\end{proposition}

\begin{proof}
The first statement is clear because any
$(-2)$-curve is disjoint from $F_0$.
For the second statement 
see~\cite[Theorem IV.5]{HM}.
\end{proof}

We now enter into the description of the 
set $\Exc(S)$ of $(-1)$-curves on $S$.
For a similar approach see~\cite{HM}.
Recall the normal bundle $\calO_{F_0}(F_0)\cong \calO_{F_0}(-K_S)$ is of order $m$ in the Picard group of $F_0$. We denote it by $\res(F_0)$. Let  $\mathfrak{pts}_0$ be the subset of $F_0$ which 
consists of points cut out by all the $(-1)$-curves of $S$ as well as their translates 
by any element of the subgroup of $\Pic(F_0)$ generated 
by $\res(F_0)$.

\begin{proposition}
\label{pro:pts}
The action of $K_S^\perp/\Lambda$
on $F_0$ by translations induces a  
transitive action of the same group on 
$\mathfrak{pts}_0$. If $\Lambda = \ker(\res)$,
then the action is free.
\end{proposition}

\begin{proof}
Let $q\in \mathfrak{pts}_0$. 
By the definition of $\mathfrak{pts}_0$
there is a translate $p$ of $q$,
by a multiple of $\res(F_0)$ which 
is cut out by a $(-1)$-curve $E$.
Given $[D]\in K_S^\perp$ the 
divisor $E+D$ has intersection 
$-1$ with $K_S$, so that $(E+D)^2$
is an odd integer by the genus formula.
Then there exists an integer $r$ 
such that $(E+D+rF_0)^2 = -1$.
By Riemann-Roch the divisor $E+D+rF_0$
is linearly equivalent to an
effective divisor $L$. Write
$L = H + R$, where $H$ contains 
all the irreducible components which 
are not contracted by $\pi$, while
$R$ contains the contracted ones,
so that $[R]\in\Lambda+\langle F_0\rangle$.
By the definition of $H$ it follows 
that it has non-negative intersection
with any $(-2)$-curve of $S$.
Since $H\cdot F_0 = 1$, it follows
that $H^2$ is an odd integer by 
the genus formula.
Then there exists an integer $s$ 
such that $(H+sF_0)^2 = -1$.
By Lemma~\ref{lem:-1} the divisor
$H+sF_0$ is linearly equivalent 
to a $(-1)$-curve $E'$.
The class of the difference $E'- (E+D)
\sim H+sF_0-(H+R-rF_0)$ is in 
$\Lambda+\langle F_0\rangle$.
So, if we denote by $p'$ the intersection
point of $E'$ with $F_0$, the class of 
the difference $p+\res(D) - p'$ is a
multiple of $\res(F_0)$. As a consequence
$p''\sim p+\res(D)$ 
is in $\mathfrak{pts}_0$.
This shows that $K_S^\perp$ acts
on $\mathfrak{pts}_0$. Since $K_S^\perp$
contains the classes of differences of 
$(-1)$-curves and the class of $F_0$,
the action is transitive.
This proves the first statement.
The second statement is clear.
\end{proof}

Let $M$ be a negative definite irreducible root lattice  of  type $A_n,D_n,E_n$, and $\alpha_1,\ldots,\alpha_n$ be its basis of simple  roots. Let $\tilde{M}$ be its extension to an affine root lattice by adding a root $\alpha_0$ such that the radical $\frakf$ of $\tilde{M}$ is generated by  $\frakf = \alpha_0 +\alpha_{\max}$, where $\alpha_{\max} = \sum_{i=1}^nm_i\alpha_i$ is the maximal root with respect to the basis of simple roots.

We extend this definition to the case when the lattice $M$ is the orthogonal sum $M_1\oplus \cdots \oplus M_k$ of irreducible root lattices  by taking the lattice $\tilde{M}$ with the radical of rank one such the quotient by the radical is $M$. 

Fix an integer valued linear function $l$ on $\tilde{M}$ and consider the convex subset in $\tilde{M}_\bbR$ defined by 
$$\Pi_M(l):= \{x\in \tilde{M}_\bbR:x\cdot \alpha_i\le l(\alpha_i),\  i = 0,\ldots,n\}.$$
For any $\alpha_i$, we have $\frakf-\alpha_i$ is a positive root (i.e. a positive integer linear  combination of simple roots). This implies that, for any $x\in \Pi_M(l)$, 
$-x\cdot \alpha_i = x\cdot (\frakf-\alpha_i) \le  l(\frakf-\alpha_i),$ hence $x\cdot 
\alpha_i\ge l(\frakf-\alpha_i)$ is bounded from below by $a_i = l(\frakf-\alpha_i)$. On the other hand, by definition, $x\cdot \alpha_i$ is bounded from above. Let $\alpha_1^*,\ldots,\alpha_n^*$ be the dual basis of $(\alpha_1,\ldots,\alpha_n)$ in $M^\vee$ (the fundamental weights). Then we can write any vector $x\in \Pi_M(l)$ as $x = x_0\frakf+\sum_{i=1}^nx_i\alpha_i^*$ and obtain, by taking the intersection with $\alpha_i$,  that 
$a_i\le t_i\le l(\alpha_i)$ and $\sum_{i=1}^nm_it_i\le l(\alpha_0)$. Obviously, $\Pi_M(l)$ is preserved under translations by the radical spanned by $\frakf$. So, its image $\bar{\Pi}_M(l) =  \Pi_M(l)/\bbR \frakf$ in the quotient 
$M_\bbR$ of $\tilde{M}_\bbR$ by this subspace is  a bounded compact rational polyhedron in $M_\bbR$.  Its intersection with the lattice $M$ consists of all roots $\alpha$ satisfying $\alpha\cdot \alpha_i\le l(\alpha_i), i = 1,\ldots,n$. Since the polyhedron is compact, this set is finite.
Note, that in the case $l \equiv  0$, $\bar{\Pi}_M(l)$ is the image of  the fundamental chamber corresponding to the root basis $(\alpha_0,\ldots,\alpha_n)$ in $M_\bbR$. It is a Coxeter polytope 
bounded by hyperplanes $x_i\le 0, \sum_{i=1}m_ix_i \ge 0$.

We apply this to our situation when the lattice $\tilde{M}$ arises as the lattice $\Lambda$ generated by $(-2)$-curves on $S$ and $[F_0]$ (the last vector is not needed if $F_0$ is reducible). Its radical is the vector $[F_0]$ and the quotient by $\la F_0\ra$ is a root lattice of finite type. Any divisor class $D$ in $\Pic(S)$, via the intersection product,  defines a function $l$ on $\tilde{M}$ as above. Copying the definition of the convex set $\Pi_M(l)$ we introduce the following definition.
 
\begin{definition}
Given $D\in\Pic(S)$ define the following 
{\em tropical Riemann-Roch space}:
\[
 L^{\rm trop}(D)
 :=
 \{R\in\Lambda\, :\, 
 \text{$(D+R)\cdot C\geq 0$
 for any $(-2)$-curve $C$ of $S$}\}.
\]
\end{definition}

The above definition is due to the fact
that $L^{\rm trop}(D)$ behaves like a 
Riemann-Roch space with respect to the 
following tropical sum: if $R = \sum_ia_iR_i$
and $R' = \sum_ib_iR_i$ are in $L^{\rm trop}(D)$
then their tropical sum
$R\oplus R' := \sum_i\max\{a_i,b_i\}R_i$
is in $L^{\rm trop}(D)$. 
 
The next proposition follows from the above discussion. 

\begin{proposition}
\label{pro:finite}
The set $L^{\rm trop}(D)/\langle mK_S\rangle$
is finite for any $D\in\Pic(S)$.
\end{proposition}

\begin{proposition}
\label{pro:bij}
Let $E\in\Exc(S)$ and let $p := \res(E)$ 
be its intersection point with $F_0$. 
Then the assignment 
$R\mapsto E+R-\frac{1}{2}(2E\cdot R+R^2)F_0$
induces an injection
\[
 \phi\colon
 L^{\rm trop}(E)/\langle mK_S\rangle
 \to
 \{E'\in \Exc(S)\, :\, \res(E'-E) \in 
 \langle\res(F_0)\rangle\}.
\]
If $\Lambda = \ker(\res)$ then 
$\phi$ is a bijection.
\end{proposition}

\begin{proof}
First of all observe that if 
$R\in L^{\rm trop}(E)$ then 
$D := E+R-\frac{1}{2}(2E\cdot R+R^2)F_0$
has non-negative intersection with all 
the $(-2)$-curves of $S$ and 
$D^2=D\cdot K_S= -1$. 
Then $D$ is linearly equivalent to a $(-1)$-curve $E'$ of $S$, by Lemma~\ref{lem:-1}.
Observe that $E'$ intersects $F_0$ at
a point $p'$ such that $p'-p\in
\langle\res(F_0)\rangle$. Then $\phi$ 
is well defined.
We now show that $\phi$ is injective.
Let $R,R'\in L^{\rm trop}(D)$ be such 
that $\phi(R) = \phi(R')$.
Then $R-R'$ is linearly equivalent to 
an integer multiple of $F_0\sim -K_S$. 
This must be a multiple of $mK_S$ because
it lies in $\Lambda$, which proves the
statement.
Assume now that $\Lambda = \ker(\res)$.
To prove the surjectivity of $\phi$, let
$E'$ be a $(-1)$-curve of $S$ such that
$\res(E'-E) = n\res(F_0)$ for some integer
$n$. Then $E'-E \sim nF_0+R$ with
$R\in\ker(\res) = \Lambda$.
In particular $E+R+nF_0\sim E'$ has
non-negative intersection with all the
$(-2)$-curves of $S$ and the same holds
for $E+R$, so that $R\in L^{\rm trop}(E)$.
\end{proof}

\begin{remark}
\label{rem:act}
The surface $S$ is acted by the group
of $\eta$-rational points of $\Pic^0(F_\eta)$,
where $F_\eta$ is the generic fiber.
This group is the homomorphic
image of $K_S^\perp$ via the pullback of the
restriction map to the generic fiber.
The kernel consists of the classes 
of vertical divisors: $\Lambda+
\langle F_0\rangle$.
Thus $K_S^\perp/\Lambda+\langle F_0\rangle$
acts on $S$. The map
\[
 \res\colon \Exc(S)\to \mathfrak{pts}_0
\]
is equivariant with respect to the above 
action. If $\Lambda = \ker(\res)$ then 
the action is free on the codomain and 
the latter is subdivided into $m$ orbits
where $m$ is as usual the index of the
rational elliptic surface.
In particular if $\Lambda$ has rank $9$
then $\Exc(S)$ is finite. To prove this 
observe first that $\mathfrak{pts}_0$ is finite,
being the union of a finite number of 
orbits for the action of a finite group.
By Proposition~\ref{pro:res} the equality
$\Lambda = \ker(\res)$ holds.
Thus, by Proposition~\ref{pro:bij},
the number of elements of $\Exc(S)$
which contain a given point $q\in
\mathfrak{pts}_0$ is in bijection with
$L^{\rm trop}(E)/\langle mK_S\rangle$,
where $E$ is a $(-1)$-curve which 
contains $q$.  The latter set  is finite
by Proposition~\ref{pro:finite}.
\end{remark}

Let $X$ be a jacobian elliptic surface. A choice of a section $E$ defines a group law on the generic fiber $F_\eta$ and hence the negation involution $a\mapsto -a$. This involution extends to a biregular involution of the surface which is the lift of the Bertini involution on the weak del Pezzo surface of degree $1$ obtained by blowing down the section $E$. Its fixed locus consists of the union of $E$ and a $3$-section of the fibration (maybe reducible).

There is analogue of this involution on any  Halphen surface $X$ of index $2$. Any $(-1)$-curve defines a $2$-section of the elliptic fibration, and as such it defines a point $x_\eta$ of degree $2$ on its generic fiber $F_\eta$. The linear system $|x_\eta|$ defines a separable degree $2$ map $F_\eta\to \bbP_\eta^1$ over the generic point $\eta$ of the base. Its deck transformation defines an involution on $F_\eta$ defined over $\eta$. The minimality of the fibration allows us to extend it to a biregular involution of the surface $X$, which we call the \emph{Bertini involution} and denote it by $\beta_E$. The scheme of the fixed points of this involution on $F_\eta$ is a reduced effective divisor of degree $4$. Its closure $B_E$ in $X$ is a curve that intersects each fiber with multiplicity $4$ (the curve $E$ is invariant but not pointwise fixed!). It intersects transversally each smooth fiber.

\begin{proposition}\label{freeaction}
The action of $\beta_E^*$ is trivial
on $L^{\rm trop}(E)$.
\end{proposition}

\begin{proof}
Observe that any $(-1)$-curve $E'$ 
which has the same restriction of $E$
to the generic fiber is preserved by 
$\beta_E$. Indeed the divisors cut 
out by $E$ and $E'$ on a general fiber
are in the same $g^1_2$ and so each
of them is mapped to itself by
$\beta_E$.
Now, for any $R\in L^{\rm trop}(E)$,
adding to the divisor $E+R$ some  multiple 
of $F_0$, we obtain a divisor $D\in L^{\rm trop}(E)$ with $D\cdot K_S = D^2 = -1$. Applying Lemma \ref{lem:-1}, we obtain that $D$ is linearly equivalent to 
a $(-1)$-curve $E'$ as above.
The statement follows.
\end{proof}

\section{The Hesse pencil} 

The Hesse pencil is the pencil of plane cubic curves of the form 
$$H_{\lambda:\mu}: \{\lambda(x^3+y^3+z^3)+\mu xyz = 0\}.$$
It has $9$ base points $x_1,\ldots,x_9$ with coordinates $(0:1:-\epsilon), (1:0:-\epsilon), (1:-\epsilon:0)$, where $\epsilon^3 = 1$. It has $4$ singular fibers, whose equations are given by (see \cite{Artebani}, \cite[3.1.3]{CAG} for this and other information about the Hesse pencil): 
\begin{align*}
H_{0:1}: \{ xyz = 0 \}, \\
H_{1:-3}: \{ (x+y+z)(x+\epsilon y+\epsilon^2 z)(x+\epsilon^2 y+\epsilon z)=0 \}, \\
H_{1:-3\epsilon}: \{ (x+\epsilon y+z)(x+\epsilon^2 y+\epsilon^2 z)(x+y+\epsilon z)=0 \}, \\
H_{1:-3\epsilon^2}: \{ (x+\epsilon^2 y+z)(x+\epsilon y+ \epsilon z)(x+y+\epsilon z)=0 \}.
\end{align*}
Each base point is an inflection point of a smooth member of the pencil. 
Each line contains $3$ base points, and each base point lies on $4$ lines. The nine points and $12$ lines form the abstract Hesse configuration $(9_4,12_3)$. All nonsingular members of the pencil have an inflection point at each base point. The polar conic of a general member of the pencil with pole at a base point $x_i$ is equal to the union of the tangent line and another line $\ell_i$ that does not depend on the choice of a nonsingular member of the pencil. It is called a \emph{harmonic polar line} of the pencil.  Each line $\ell_i$ intersects a general member with multiplicity $3$.  It passes through one of the singular points of a singular member and intersects the opposite side of the triangle at one point. The $12$ singular points of the four triangles and $9$ harmonic polar lines form an abstract configuration $(12_3,9_4)$ dual to the Hesse configuration. 

The blow-up $\pi \colon S\to \bbP^2$ of the base points $x_1,\ldots,x_9$ is a rational elliptic surface $f \colon S\to \bbP^1$ of index $1$. The exceptional curves $\calE_i = \pi^{-1}(x_i)$ are sections of the fibration. As we remarked earlier, they correspond to rational points of the elliptic curve $F_\eta$ over the field $K = \Bbbk(t)$ of rational functions on the base $\bbP^1$ of the fibration. Fixing one base point, say $x_1$ and the corresponding section $E_1$, and hence a rational point on $F_\eta$, we equip $F_\eta$ with a group law.  It defines a group law on any nonsingular fiber $F$  with the zero point $F\cap E_1$. It also defines a group law of a one-dimensional commutative algebraic group on the set of nonsingular points of each fiber. The group of rational points on $F_\eta$ is a finitely generated abelian group, called the \emph{Mordell-Weil group} of the elliptic surface. Any element of this group defines a translation automorphism of $F_\eta$ that extends to a biregular automorphisms of $S$. In our case the 9 sections generate the Mordell-Weil group isomorphic to $(\bbZ/3\bbZ)^{\oplus 2}$.

\begin{definition} An \emph{Halphen pencil of Hesse type} is an Halphen pencil defining a rational elliptic surface (\emph{Halphen surface of Hesse type}) with the jacobian fibration isomorphic to the elliptic fibration defined by the Hesse pencil.
\end{definition}

Note that the correspondence between relatively minimal rational elliptic surfaces and  Halphen pencils of index $m$ is far from being bijective since it depends on a choice of a birational morphism $\pi \colon S\to \bbP^2$.  Two Halphen pencils corresponding to the same elliptic surface differ by a Cremona transformation of the plane.

\section{The Chilean configuration of conics} \label{chileconf}

Our goal in this  section is to give an explicit construction of an Halphen pencil of Hesse type and index $2$ such that its four singular members are unions of three conics.

Although we know that there exists a rational elliptic surface of index $2$ with the jacobian surface defined by the Hesse pencil, it is not obvious that we can find a morphism $\pi \colon S\to \bbP^2$ such that the images of the singular fibers are the unions of three conics (other possibility is the union of two lines and a plane quartic). 

To do so we first fix a nonsingular plane cubic $C: \{G_3 = 0\}$. Then we choose $9$ distinct points such that they can be partitioned in $4$ different ways in three subsets $(x_i,x_j,x_k)$ of $3$ points such that $(x_i+x_j+x_k-h) \sim 0$. In this way we have $12$ triples. Note that the base points on each line component $x+\epsilon y+\epsilon^2 z = 0$ of a singular member of the Hesse pencil add up to $0$ in the group law on a fixed nonsingular member of the pencil with the zero point defined by a choice of one of the base points $x_i$. We can parameterize the set of line components of each reducible fiber by $\bbF_3$ in such a way that each section is represented by a vector $(a_1,a_2,a_3,a_4)$ indicating which component of a fiber it intersects. This identifies the set of sections with the plane $\bbF_3^2$ embedded in $\bbF_3^4$ in such a way that  the difference of two vectors has only one zero coordinate \cite{MirandaPersson}. For example, we may assume that

\begin{table}[h]
\centering
\begin{tabular}{|c|c|c|c|c|c|}
\hline
$x_1$&(0,0,0,0)&$x_4$&(1,1,0,2)&$x_7$&(2,2,0,1)\\
$x_2$&(0,1,2,1)&$x_5$&(1,2,2,0)&$x_8$&(2,1,1,0)\\
$x_3$&(0,2,1,2)&$x_6$&(1,0,1,1)&$x_9$&(2,0,2,2)\\ \hline
\end{tabular}
\label{table1}
\end{table}

Now we identify the 12 line components of reducible fibers with lines in the finite affine plane $\bbF_3^2$. Three points lie on a line if and only if the sum of the corresponding $4$-vectors is equal to $0$. This gives a finite projective geometry realization of the  Hesse configuration. Note that $3$ points in a line add up to zero in the group law on a nonsingular member $F = F_t$ of the pencil with the zero point equal to $x_1$. The four lines containing the origin $x_1$ are given by the first column of the previous table, the first row, the diagonal and $x_1,x_6,x_8$.

 We can order the coordinates of the nine base points of the Hesse pencil as follows:
\begin{gather}
x_1 = (0:1:-1), \quad x_2 = (0: 1:-\epsilon), \quad x_3 = (0:1:-\epsilon^2), \notag\\
x_4 = (1:0:-1), \quad x_5 = (1: 0:-\epsilon^2), \quad x_6 = (1:0:-\epsilon),\notag \\
x_7 = (1:-1:0), \quad x_8 = ( 1:-\epsilon:0), \quad x_9 = (1:-\epsilon^2:0),\label{points}
\end{gather}
where $\epsilon$ is a primitive 3rd root of unity.

Let $F[3]$ be  the group of $3$-torsion points on the curve $F$ with the point $x_1$ as the zero element. Then the lines of the Hesse configuration that contain $x_1$ contain 3 points which form a cyclic subgroup of $F[3]$. Each line in the same member of the pencil is identified with a coset with respect to this subgroup. 

Now fix a non-trivial $2$-torsion $\tau$ point on $F$. Consider 9 points $y_i$ equal to  $x_i+ \tau$ in the group law. Then each of the 12 cosets define a set of 3 points on $F$ added up to $\tau$ in the group law. We have partitioned the 9 points in 12 different ways as the union of three triples of points added up to $\tau$ in the group law. Take two cosets  $H_i,H_j$ with no common elements. Then their union defines a set of 6 points whose sum in the group law is equal to zero. Take 3 cosets $H_i,H_j,H_k$ defining a singular fiber of the Hesse pencil. Adding up $\tau$, we obtain 3 disjoint sets $H_i+\tau, H_j+\tau, H_k+\tau$ of triples on $F$ that  add up to $\tau$. It follows that there exists a conic  that intersects $F$ at 6 points from two of these sets. The sum of these three conics pass through 9 points of the Halphen pencil with multiplicity 2 at each point. This defines a singular fiber of the Halphen pencil equal to the union of 3 conics. We have 4 singular fibers, hence 12 conics. Each base point defines a 2-section of the elliptic fibration that intersects 2 components of a singular fiber. This means that each base point lies on 8 conics, two from each fiber. Moreover, each conic passes through 6 points, hence we have constructed the Chilean configuration $(12_8,9_6)$ of 12 conics and 9 points in the plane.

\begin{remark} $F[3]$ is clearly isomorphic to the affine plane over the prime field $\mathbf F_3$ which can be completed to the projective plane, with the line at infinity distinguished. Components of reducible fibers corresponding to 'finite' lines in the projective plane, two components belong to the same fiber if and only if they are parallel (meaning cosets of the same cyclic subgroup), i.e. meeting at the line at infinity, whose points hence correspond to the reducible fibers, each given by the pencil of lines through the points. Points will correspond to sections, and we see how two components can only share one section and each section will intersect exactly one component of each fiber. Note that with a proper choice of zero, lines in the finite projective planes correspond to bona fide lines in the complex setting, but for the combinatorial codification this is not important. In fact we can by a slight abuse of notation refer to the same finite plane when we are talking about the translated points in the index two setting. Then three points on a line, does no longer mean that they are literally on a line. This extended use will turn out to be useful further down.  Thus any two parallel lines in the new interpretation can be associated to a conic component of the Halphen pencil.
\end{remark}

We conclude the section with some
observation of the automorphism group
of $S$. First of all any birational automorphism that preserves the fibration is a biregular automorphism. This follows from the relative minimality assumption since no horizontal nor vertical curve can be blow down.  An automorphism of $S$ preserves the pencil but may act non-trivially on it. The subgroup of automorphisms that act trivially on the base (which could be called \emph{vertical}) is easy to describe. It is mapped to the group of automorphisms of the jacobian fibration with the kernel equal to the group of translations of $S$ by the elements of the Mordell-Weil group of the jacobian fibration. The group of automorphisms of the jacobian fibration is the group of automorphisms of its generic fiber defined over the field of rational functions of the base. It is the semi-direct product of the group of Mordell-Weil group (defined when we fix a section) and the group of automorphisms fixing a section. The latter depends on the  $j$-invariant of the elliptic curve but always contains the Bertini involution.  When $m = 2$, as we explained in Section 2, we have similar involutions, the Bertini involutions $\beta_E$. We refer for the study of  automorphism groups of Halphen surfaces to \cite{DolgachevMartin}.

Recall that the Hesse group $G_{216}$ is one of maximal finite groups of projective transformations of the plane. It can be realized as the group of automorphisms of the Hesse pencil of cubic curves, or, equivalently, as the group of automorphisms of a rational  Halphen surface of Hesse type and of index 1 (see \cite{Artebani} or \cite{CAG}). The group $G_{216}$ has a normal subgroup $H$ isomorphic to $(\bbZ/3\bbZ)^2$ and the quotient is isomorphic to the binary octahedral group $2.\frakA_4$ which is isomorphic to $\SL(2,\bbF_3)$. The pre-image of the central subgroup of $\SL(2,\bbF_3)$ under the quotient map is the subgroup $G_{18}$ isomorphic to the semi-direct product $(\bbZ/3\bbZ)^2\rtimes (\bbZ/2\bbZ) \cong \bbZ/3\bbZ\rtimes\frakS_3.$
 It consists of automorphisms generated by translation and any of the nine Bertini involutions defined by the negation automorphism of a general fiber when we fix the structure of group on it by fixing one of the sections. The projection to $G_{168}\to \frakA_4$ corresponds to the action on the base of the elliptic fibration as the octahedral group of automorphisms of $\bbP^1$. 

\begin{corollary} 
\label{aut}
The automorphism group of an 
Halphen surface of index $2$
of Hesse type is isomorphic to the subgroup $G_{18}$ of the Hesse group $G_{216}$ of projective transformations of $\bbP^2$. Its normal subgroup of order $9$ consists of translation automorphism by elements of the Mordell-Weil group of the jacobian fibration. Its $9$ involutions are the Bertini transformations.
\end{corollary}

\begin{proof}  An ordered Chilean set $\calE$ of $(-1)$-curves  defines a geometric basis $(e_0,e_1,\ldots,e_9)$ of $\Pic(S)$. By Corollary \ref{invariant}, $e_1+\cdots+e_9$ is invariant with respect to the action of $\Aut(S)$ on $\Pic(S)$, and since $K_S = 3e_0-(e_1+\cdots+e_9)$ is obviously invariant too, we obtain that the divisor class $e_0$ is invariant. This shows that any $g\in \Aut(S)$ descends to an automorphism $T$ of $\bbP^2$ under the blowing down morphism $\pi:S\to \bbP^2$ defined by the linear system $|e_0|$. Thus we can identify $\Aut(S)$ with a group of projective transformations of $\bbP^2$. Composing any $g\in \Aut(S)$ with the translation automorphism $t_a$ for some element of the Mordell-Weil group, we may assume that $g$ fixes a $(-1)$-curve $E_1\in \calE$. Then it commutes with the Bertini involution $\beta_{E_1}$, descends to the del Pezzo surface $\calD$  obtained by blowing down $E_1$, hence fixes the point $s'$ on the image of $F_0$ on $\calD$ equal to the unique base point of the anti-canonical linear system $|-K_{\calD}|$. We identify this point with a point on $F_0$. Thus $T$ fixes 4 points on the image $\bar{F}_0$ of $F_0$ in the plane:the images of $s',s = E_1\cap F_0,$ and the images of the fixed points of $\beta_{E_1}$ on $F_0$.  

 Since a projective transformation cannot pointwise fix a cubic curve, $g^2$ has 4 fixed points on $F_0$ and hence is either identity or coincides with an involution of $F_0$ with 4 fixed point. Since $\beta_{E_1}$  fixes the same points on $F_0$, $g^2 = \beta_{E_1}$. It follows that $g$ has two fixed points on $E_1$:one is $s$ and an another is $s_0 = F_1\cap E_1$, where $F_1$ is some other fiber. Since $E_1$ intersects $F_1$ with multiplicity $2$, it must be tangent to $F_1$ at the fixed point $s_0$. This is obviously impossible since  the blow-up of $s_0$ would have  a fixed point with three different tangent directions defined by $F_1,E_1$ and the exceptional curve, hence acts identity at the tangent space of the surface at this point is hence identity (because we assume that characteristic is different from 2).  Thus $g$ coincides with $\beta_{E_1}$ and this proves the assertion.
\end{proof}

\begin{remark} The same proof works for degenerate Chilean configuration when $2F_0$ is a double fiber. In fact, it is even easier. The group of automorphisms contains $G_{18}$ and leaves invariant a triangle of lines in the plane. This means that the group $G$ of automorphisms is an imprimitive linear group, i.e. it contains a normal subgroup that arises from a reducible linear representation. The known classification of finite subgroup of projective automorphisms shows that this group coincides with $G_{18}$.
\end{remark}

\section{Explicit formulas and uniqueness}

We take the harmonic polar line $\ell_1$ with respect to the point $x_7 = (1:-1:0)$. It has equation $x-y = 0$. It intersects the nonsingular cubics $H_{1:t}$ at three $2$-torsion points $(1:1:a)$, where $a$ is a root of the cubic equation $X^3+tX+2 = 0$. This gives us 9 base points of the Halphen pencil of Hesse type and index $2$:

\begin{align*}
p_1 &= (a : 1 : 1),  & p_2 & = (\epsilon a : \epsilon^2 : 1), & p_3 &= (\epsilon^2 a : \epsilon : 1), \notag\\
p_4 &= (1 : a : 1),  & p_5 & = (\epsilon : \epsilon^2 a : 1), & p_6 &= (\epsilon^2 : \epsilon a : 1),\notag \\
p_7 &= (1 : 1 : a),  & p_8 & = (\epsilon : \epsilon^2 : a), & p_9 &= (\epsilon^2 : \epsilon : a)
\end{align*}

Given the nine base points it is easy to
compute the classes of the $(-2)$-curves
of $S$. These are all the conics through 
six of the nine points whose sum is zero.
The strict transforms of these conics are
the irreducible components of the four 
reducible fibers of type $I_3$ of the following 
pencil of plane sextics: 

\beq\label{eqpencil}
   x^3y^3 + x^3z^3 + y^3z^3 
   -\frac{1}{a}(x^4yz + xy^4z + xyz^4)
   +\frac{1-a^3}{a^2}x^2y^2z^2 
    + \lambda
    \left(x^3 + y^3 + z^3 -\frac{a^3+2}{a}xyz\right)^2 = 0.
\eeq

Recall that for any smooth plane cubic curve $C \subset \bbP^2$, the dual curve $C^*$ is a plane sextic in the dual plane $\check{\bbP}^2$ with 9 cusps. If one considers a plane cubic $F_3\subset \check{\bbP}^2$ passing through these cusps, we can define a pencil of plane sextic curves spanned by $C^*$ and  $F_3$. It is an example of an Halphen pencil of index $2$. We say that this pencil is generated by the dual plane cubic $C^*$. 

Recall from \cite[3.2]{CAG} that the \emph{Caylean} of a plane cubic is the locus of lines $\la p,q\ra$ that are line components of reducible polar conics.  

\begin{proposition}\label{dualcubic} 
The Halphen pencil of Hesse type in equation \eqref{eqpencil} is generated by $C^*$, where $C=\{u^3 + v^3 + w^3 - 3auvw = 0\}$, and the multiple fiber is the Caylean of $C$.

\end{proposition}

\begin{proof} We can rewrite equation \eqref{eqpencil} using only the parameter $a$ 
 instead
of $a$ and $t$:
\[
 (xy - az^2)(xz-ay^2)(yz-ax^2)
    + \lambda
    (a(x^3 + y^3 + z^3) -(a^3+2)xyz)^2 = 0.
\]
The pencil contains a unique sextic with
nine cusps. Its equation is the following:
\[
 x^6 +  y^6 + z^6 +
 (4a^3 - 2)(x^3y^3 + x^3z^3 + y^3z^3)
 -6a^2(x^4yz + xy^4z + xyz^4)
 -3a(a^3 - 4)x^2y^2z^2 = 0.
\]
Its dual is the plane cubic $C$ of equation in dual coordinates $u,v,w$ is 
\[
 u^3 + v^3 + w^3 - 3auvw = 0 
\]
 (see \cite[3.2.3]{CAG}). The equation of the Caylean of this curve is \cite[(3.27)]{CAG} $$x^3+y^3+z^3-\frac{2+a^3}{a}xyz = 0.$$ Since $t = -\frac{2+a^3}{a}$, we see that the cubic $$x^3+y^3+z^3+txyz = 0$$ coincides with the Caylean of the cubic $C$ and our pencil coincides with the pencil generated by the dual $C^*$ of $C$.
\end{proof}

\begin{remark} We are not writing here the values
of $\lambda\ne 0,\infty$ which correspond to the special
fibers because
these values $\lambda_1,\lambda_2$ are nasty. One can check that their cross ratio $R$ satisfies the equation $R^2-R+1 = 0$ that means that the double cover of $\bbP^1$ branched over $0,\infty,\lambda_1,\lambda_2$ is an elliptic curve with the absolute invariant $j$ equal to zero as it happens in the case of the jacobian fibration.  It is known that the invariant of binary quartics $a_0u^4+a_1u^3v+a_2u^2v^2+a_3uv^3+a_4v^4$ that vanishes on the set of quartics whose zeroes are four points on $\bbP^1$ with such cross-ratio (\emph{equianharmonic quartics}) is equal to $S = a_0a_4-3a_1a_3+4a_2^2$. When we fix the zeroes $0$ and $\infty$ corresponding to vanishing of the coefficients $a_0,a_4$, we see that  the set of unordered pairs $\{\lambda_1,\lambda_2\}\in (\bbP^1)^{(2)}\cong \bbP^2$ is a conic $-3a_1a_3-4a_2^2 = 0$. We have to throw away 4 points from this conic corresponding to $\lambda_1,\lambda_2\in \{0,\infty\}$ or $\lambda_1 = \lambda_2$. This should be compared with Remark \ref{moduli} below where we show that the parameter $t$ of our Halphen pencils of Hesse type take value in $\bbP^1\setminus \{4 \textrm{points}\}$. This shows that the map that assigns to a parameter $t$ the unordered pair $\{\lambda_1,\lambda_2\}$ is bijective.
\end{remark}

The classes of the $(-2)$-curves of the surface 
$S$ expressed in a geometric basis by the 
columns of the following matrix

\beq\label{matrixofconics}
 \begin{pmatrix*}[r]
2&2&2&2&2&2&2&2&2&2&2&2&\\
-1&-1&0&-1&-1&0&-1&-1&0&-1&0&-1&\\
-1&-1&0&-1&0&-1&-1&0&-1&-1&-1&0&\\
-1&-1&0&0&-1&-1&0&-1&-1&0&-1&-1&\\
-1&0&-1&-1&-1&0&-1&0&-1&0&-1&-1&\\
-1&0&-1&-1&0&-1&0&-1&-1&-1&0&-1&\\
-1&0&-1&0&-1&-1&-1&-1&0&-1&-1&0&\\
0&-1&-1&-1&-1&0&0&-1&-1&-1&-1&0&\\
0&-1&-1&-1&0&-1&-1&-1&0&0&-1&-1&\\
0&-1&-1&0&-1&-1&-1&0&-1&-1&0&-1&\\
 \end{pmatrix*}.
\eeq
For example, the first column exhibits a conic passing through the first six points $p_1,\ldots,p_6$. 
The reducible fibers are the sums of columns $(1,2,3), (4,5,6), (7,8,9)$ and $(10,11,12)$.
The defining polynomials for the above 
$12$ conics are the following:

\begin{align*}
    xy - az^2,\\
    xz - ay^2,\\
    yz - ax^2,\\
    x^2 + (\epsilon a + \epsilon )xy + \epsilon^2 y^2 + (\epsilon^2 a +\epsilon^2 )xz + (a + 1)yz 
        + \epsilon z^2,\\
    x^2 + (\epsilon^2 a +\epsilon^2 )xy + \epsilon y^2 + (\epsilon a + \epsilon )xz + (a + 1)yz + 
    \epsilon^2z^2,\\
    x^2 + (a + 1)xy + y^2 + (a + 1)xz + (a + 1)yz + z^2,\\
    x^2 + (\epsilon a + 1)xy + y^2 + (\epsilon^2 a + \epsilon )xz + (\epsilon^2 a + \epsilon )yz + 
        \epsilon^2 z^2,\\
    x^2 + (\epsilon^2 a + \epsilon )xy + \epsilon^2 y^2 + (\epsilon a + 1)xz + (\epsilon^2 a + 
        \epsilon )yz + z^2,\\
    x^2 + (a +\epsilon^2 )xy + \epsilon y^2 + (a +\epsilon^2 )xz + (\epsilon^2 a + \epsilon )yz + 
        \epsilon z^2,\\
    x^2 + (\epsilon a +\epsilon^2 )xy + \epsilon y^2 + (\epsilon^2 a + 1)xz + (\epsilon a +\epsilon^2 )yz +
        z^2,\\
    x^2 + (a + \epsilon )xy + \epsilon^2 y^2 + (a + \epsilon )xz + (\epsilon a +\epsilon^2 )yz + 
    \epsilon^2z^2,\\
    x^2 + (\epsilon^2 a + 1)xy + y^2 + (\epsilon a +\epsilon^2 )xz + (\epsilon a +\epsilon^2 )yz + 
        \epsilon z^2
\end{align*}

\begin{remark}\label{degChileConf} 
When we choose the parameter $t$ to be $\infty$ or satisfy $t^3+27 = 0$, the cubic curve $\{x^3+y^3+z^3+txyz = 0\}$ becomes the union of three lines. The corresponding cubic equation $X^3+tX+2 = 0$ has two roots, one of them is a double root. In the case of the root of multiplicity $1$, we get a pencil of sextics with the corresponding $9$ distinct base points $p_1,\ldots,p_9$ as above, which lie on the triangle $H_{\lambda:\mu}$. We call the corresponding configuration of $9$ conics and $3$ lines the \emph{degenerated Chilean configuration}. In the construction of the pencil as a torsor of the jacobian fibration, the first parameter corresponds to one choice of a non-trivial $2$-torsion element of the connected component of the Picard scheme of the singular fiber isomorphic to the multiplicative group $\bbG_m$. 

The pencil corresponding to the double root $a$ of $X^3+tX+2 = 0$ is a pencil of elliptic curves with $3$ triple points $x_1,x_2,x_3$ (specialization of the nine $p_i$'s above) and $9$ simple infinitely near points 
$x_i^{(1)}\succ x_i,\ x_i^{(2)}\succ x_i, x_i^{(3)}\succ x_i, i = 1,2,3$, where $\succ$ indicates an infinitely near point of order $1$. For example, if we take $t =-3$, then the multiple root $a$ is equal to $1$ and the pencil is given by equation 
$$(x^2-yz)(z^2-xy)(y^2-xz)+\lambda (x^3 + y^3 + z^3 -3xyz)^2 = 0.$$ Each of the conics in the fiber over $\lambda = 0$ pass through three base points $x_1,x_2,x_3$ and intersect each other transversally at one of the points $(1:0:0), (0:1:0), (0:0:1)$. The base points are also the singular points of the cubic that enters with multiplicity 2. If we apply a quadratic Cremona transformation with fundamental points at the base points, the pencil is transformed to an Halphen pencil of Hesse type of index 1.  In fact, one can prove that the Halphen surfaces of Hesse type from the pencil in \ref{eqpencil} have four degenerations (one for each triangle and corresponding double root $a$) into a singular surface with three singularities, which has an elliptic fibration with four $I_3$ fibers such that at the $3$ vertices of one of them we have cyclic quotient singularities of type $1/4(1,1)$ (i.e. germ at $(0,0)$ of the quotient of $\mathbb C^2$ by $(x,y) \mapsto (ix,iy)$). The construction of the singular surface is the following. Take the jacobian Hesse fibration, and consider one of the four $I_3$ fibers (this is chosen by the degeneration), blow-up the $3$ nodes and contract the $3$ proper transforms from this $I_3$ fiber. Hence we have that Halphen surfaces of Hesse type degenerate into a singular model of their jacobian fibration. This degeneration is in fact a $\mathbb Q$-Gorenstein smoothing, this is, the canonical class of the corresponding $3$-fold is $\mathbb Q$-Cartier.

Roughly speaking, the way one constructs the degeneration is backwards. We take the Hesse surface, and construct the singular surface described above. One can prove that this surface 
has no local-to-global obstructions to deform (e.g. as in \cite[Section 4]{PSU}). In fact, this is also true after contracting three pairs of $(-2)$-curves from the rest of the $I_3$ fibers (one can use \cite[Theorem 4.4]{PSU}). So we do contract them as well, and consider a $\mathbb Q$-Gorenstein degeneration keeping these three $A_2$ singularities and producing a new $A_2$ singularity from the singular $I_3$ fiber. The later is possible because we can reduce that deformation to a deformation of $1/12(1,5)$ which has a $1$ parameter deformation into an $A_2$ singularity (see e.g. \cite[Proposition 2.3]{HP}). It can be proved that the general fiber is a Halphen surface of index $2$ keeping an elliptic fibration from the one in the Hesse surface, and since we have $4$ $A_2$ singularities with not two in one fiber, we have no more room for other Halphen surface than the Halphen surface of Hesse type. Now since we know that there is just one family of them, this corresponds to the Chilean elliptic pencil. This phenomena may be general, it could happen for any jacobian fibration with a $I_n$ fiber: Any Halphen surface on index $m$ can $\mathbb Q$-Gorenstein degenerate to a singular model of its jacobian.

We call the Halphen pencil corresponding to the choice of a reducible fiber $F_t$ and the non-multiple root of $X^3+tX-2 = 0$ the \emph{degenerate Halphen pencil of Hesse type}. It can be obtained from an  Halphen pencil that defines the degenerated Chilean configuration.
\end{remark}

 \begin{remark} \label{moduli}
 We already know from explicit formulas that a Chilean configurations of conics depends on one parameter. We will prove later in Proposition \ref{prop:unique} that there is only one Chilean pencil (up yo projective equivalence) that gives rise to an Halphen surface of index 2  of Hesse type. Thus the moduli space of Chilean configurations (including degenerate ones) is the same as the moduli space of Halphen surfaces of index 2 of Hesse type. To describe it more precisely, we have to invoke the theory of torsors of jacobian elliptic surfaces (see \cite[Chapter 4]{CDL}). In the case when the jacobian elliptic surface $j:J\to \bbP^1$ is rational it gives that the isomorphism class of  a torsor of index $m$ is determined uniquely by the data that consists of points $x_1,\ldots,x_k$ on the base and non-trivial $m_i$-torsion points in the connected group of the Picard variety $\mathbf{Pic}(J_{x_i})$ of the fibers $J_{x_i}$. The corresponding torsor is an elliptic fibration with $m_i$-multiple fibers over the points $x_i$. In the special case when all $m_i$ are equal to $2$, so that the index of the torsor is equal to $2$, one can make it more explicit by considering the Weierstrass model of the jacobian fibration. It defines a degree $2$ map from $J$ to the minimal ruled surface $\bfF_2$ whose branch divisor is the union of the special section $E$ with $E^2 = -2$ and  a curve $B$ disjoint from $E$ from the linear system $|6f+3e|$, where $f$ is the divisor class of the ruling and $e$ is the divisor class of $E$. A choice of a nonsingular point $b_i$ on $B$ defines a point $x_i$ on the base and a non-trivial $2$-torsion element in $\mathbf{Pic}_{J_{x_i}}$. Thus the variety of torsors of index $2$ with fixed number of double fibers is isomorphic to the symmetric product $(B^\sharp)^{(k)}$, where $B^{\sharp}$ is the set of smooth points of $B$. Applying this to our situation we find that the pre-image of $B$ on $J$ is a $3$-section of the elliptic fibration that intersects each singular fiber at one of its singular points. Its image in the plane is a harmonic polar line. Thus the moduli space can be identified with  a harmonic polar line minus four singular points of the singular fibers lying on it.

One can also apply the quadratic twist construction of the torsor. Fix a section and consider the Weierstrass model as above. The image of the harmonic polar line on it is a curve of arithmetic genus 4 with 4 cusps. Choose a twisted cubic $K$ on $\bfF_2$ (i.e an irreducible curve from the linear system $|3f+e|$) that passes through the four cusps and passes through some simple point $b$ of $B$. Since $\dim |3f+e| = 5$, we can do it. The curve $K$ intersects the exceptional section $E$ at one point $b'$.

The pre-image of $K$ on $J$ is a smooth rational $2$-section $R$ that is tangent to a fiber at a point $\tilde{b}$ and passes through the point $\tilde{b}'$ on the proper transform of $E$, the zero section  of $j \colon J\to \bbP^1$.  Let $X \to J$ be the double cover of $J$ ramified over the two fibers at $t= j(\tilde{b})$  and $t' = j(\tilde{b}')$. It follows from the formula for the canonical class of a double cover that its canonical class is trivial. For simplicity, we assume that $J_{t}$ and $J_{t'}$ are smooth. This implies that $X$ is a smooth K3 surface (otherwise we have to replace $X$ by its minimal resolution). The pre-image of $R$ in $X$ splits into two sections $R_+$ and $R_-$ intersecting at two points transversally, one point over $\tilde{b}'$ where it intersects the pre-image $\mathsf{O}$ of the zero section. Now we define an involution $\tau$ on $X$ that is the composition of the deck transformation of the double cover $X\to J$ and  the translation $\iota_{\mathsf{O}}$ by $\mathsf{O}$, where we take $R_+$ as the zero section. We have $\tau$ acts identically on $X_{t'}$ and as a translation by a $2$-torsion point on $X_t$. The quotient is our surface $S$.  
\end{remark}

\begin{lemma}
\label{eight}
Let $S$ be an Halphen surface of Hesse type and index $2$. Assume that the half-fiber is not one of the four triangles. Then $S$ contains a $(-1)$-curve which intersects eight $(-2)$-curves.
\end{lemma}

\begin{proof}
Let $E$ be a $(-1)$-curve of $S$.
If $E$ intersects all the four triangles 
in two sides there is nothing to prove.
Assume this is not the case and let 
$T_1,\dots,T_k$, with $1\leq k\leq 4$
be the triangles which are intersected 
by $E$ at just one side and let $R_i^{(1)},
R_i^{(2)}$ be the two curves in $T_i$ which
are disjoint from $E$. Let
\[
 E' := E - \sum_{i=1}^k (R_i^{(1)}+R_i^{(2)}) + kF_0.
\]
Observe that $E'^2 = E'\cdot K_S = -1$,
so that by Riemann-Roch 
$E'$ is linearly equivalent to an 
effective divisor which, with abuse of notation,
we will denote by the same letter.
Moreover $E'$ has non-negative 
intersection with any of the twelve
$(-2)$-curves of $S$, so, by Lemma \ref{lem:-1}, $E'$
is a $(-1)$-curve of $S$. 
Finally $E'$ intersects each triangle
along two sided as desired.
\end{proof}


Next theorem is a uniqueness result. It shows that, up to projective equivalence, the family of Chilean configurations from Equation \ref{eqpencil}   
represents all Halphen pencils of Hesse type of index $2$.

\begin{theorem}
\label{unique}
Let $S$ be an Halphen surface of index $2$
whose jacobian is the Hesse surface. Assume that the half-fiber is not one of the four triangles.
Then, $S$ is a Chilean surface, this is, it comes from a Chilean configuration via the pencil in Equation \ref{eqpencil}.
\end{theorem}

\begin{proof}
Let $E$ be a $(-1)$-curve as in 
as in Lemma~\ref{eight}, let 
$R_i^{(1)}, R_i^{(2)}$ be the two sides 
of the triangle $T_i$ which are
intersected by $E$ and let $R_i^{(0)}$
the remaining side.
Observe that we are free to (re)label 
the four singular fibers as well as the
components in each such fiber as we 
wish.
Let $\Lambda\subseteq
K_S^\perp$ be the sublattice spanned 
by the classes of the $(-2)$-curves, as before.
Let $\sigma\in\Aut(S)$ be such that
$\sigma(E)\neq E$. Observe that $\sigma(E)$
and $E$ does not have the same intersection
product with any $(-2)$-curve of $S$. Indeed,
if this were the case, the class of
$\sigma(E)-E$ would be proportional to
$F_0$, but the only divisor of the form $E+nF_0$ with self-intersection $-1$ is
$E$, a contradiction.
The $(-2)$-classes
generate $K_S^\perp$ over the rationals
and $\sigma(E)-E$ belongs to this space,
so that there exists rational numbers $a_{ij}$
such that
\[
 \sigma(E)-E \sim
 \sum_{i=1}^4 (a_{i0}R_i^{(0)} + a_{i1}R_i^{(1)} + a_{i2}R_i^{(2)}).
\]
Let $i\in\{1,2,3,4\}$ be the index
of a singular fiber where $E$ and 
$\sigma(E)$ have different intersection
numbers with the three irreducible
components.
After possibly switching the labels of 
$R_i^{(1)}$ and $R_i^{(2)}$ we can assume
$\sigma(E)\cdot R_i^{(0)}=1$, $\sigma(E)\cdot R_i^{(1)}=0$ and 
$\sigma(E)\cdot R_i^{(2)}=1$. These equations
imply $a_{i0}=-\frac{1}{3}+n, 
a_{i1}=\frac{1}{3}+n$ and $a_{i2}=n$,
where $n\in\mathbb Z$. In other words
the class of the $\mathbb Q$-divisor
\[
 \frac{1}{3}(R_i^{(1)}-R_i^{(0)})
\]
must appear in $\sigma(E)-E$.
Since the above divisor has self-intersection
$-2/3$ the difference $\sigma(E)-E$ must contain 
exactly the sum of three such divisors, whose
self-intersection is $-2$. 
Moreover we can assume without loss of 
generality that the sum is supported
at the last three singular fiber, so that
it is
\[
 D_{0111} := 
 \frac{1}{3}(R_2^{(1)}-R_2^{(0)}+
 R_3^{(1)}-R_3^{(0)}+R_4^{(1)}-R_4^{(0)}).
\]
Here the notation $D_{0111}$ means 
that the divisor is supported at the 
first three singular fibers and in each
such fiber we are taking $R_i^{(1)}-R_i^{(0)}$
instead of $R_i^{(2)}-R_i^{(0)}$.
We showed that the class of $D_{0111}$ is in $\Pic(S)$ but the sublattice
$\Lambda+\langle D_{0111},F_0\rangle/\langle F_0\rangle
\subseteq K_S/\langle F_0\rangle$
is still not unimodular, so that
another class has to be added to get
the full lattice $K_S$. Acting with another automorphism $\tau$ which comes
from the jacobian action we get a
$(-1)$-curve $\tau(E)$ distinct from
$E$ and $\sigma(E)$.
Reasoning as above with $\tau(E)-E$ one
concludes that the new class to be
added is again the sum of three 
divisors of type
$\frac{1}{3}(R_i^{(a)}-R_i^{(0)})$, where 
$a\in\{1,2\}$. Moreover the new class must have integer intersection product
with $D_{0111}$.
Since $\frac{1}{3}(R_i^{(1)}-R_i^{(0)})\cdot 
\frac{1}{3}(R_i^{(2)}-R_i^{(0)}) = -\frac{1}{3}$,
the only possibilities for the new
class are $D_{0222}$, $D_{a012}$, 
$D_{a021}$, $D_{a210}$, $D_{a120}$,
$D_{a201}$, $D_{a102}$, where $a\in\{1,2\}$. 
The first possibility cannot occurr
because $D_{0222}\sim
2F_0-R_2^{(0)}-R_3^{(0)}-R_4^{(0)}-D_{0111}$ is
already in $\Lambda+\langle
F_0,D_{0111}\rangle$.
Without loss of generality we can 
relabel the last three fibers and 
relabel the components in the first 
fiber because these operations leave
$D_{0111}$ unaltered. In this way
we can assume the second divisor to be
\[
 D_{1012} :=
 \frac{1}{3}(R_1^{(2)}-R_1^{(0)}+
 R_3^{(1)}-R_3^{(0)}+R_4^2-R_4^{(0)}).
\]
Observe that $D_{0111}\cdot D_{1012}
= -1$ and each of the two divisors
has intersection $1$ with $E$.
The class of each of the following divisors
is in the Picard group because it is a linear
combination with integer coefficients of 
$(-2)$-curves, $F_0$, $E$, $D_{0111}$ and 
$D_{1201}$
\[
 \arraycolsep=1cm
 \begin{array}{lll}
  E_1 := E & E_2 \sim E+D_{0222} 
  & E_3 \sim E+D_{0111}\\
  E_4 \sim E+D_{2021} & E_5 \sim E+D_{2210} 
  & E_6 \sim E+D_{2102}\\
  E_7 \sim E+D_{1012} & E_8 \sim E+D_{1201} 
  & E_9 \sim E+D_{1120}.
 \end{array}
\]
Each such divisor has non negative
intersection with any $(-2)$-curve, it has
self-intersection $-1$ and intersection $-1$
with $K_S$, so it is linearly equivalent to a
$(-1)$-curve by Lemma~\ref{lem:-1}.
Moreover $E_i\cdot E_j = 0$ for any
$i\neq j$.
The orthogonal complement of the classes
of these nine $(-1)$-curves is generated by
the class of 
\[
 H := 3E + 3F_0 -\sum_{i=1}^4R_i^{(0)},
\]
which has self-intersection $1$ and
intersection product $2$ with any 
$(-2)$-curve of the surface.
Thus all the $(-2)$-curves are conics 
in the plane model defined by $H$.
To conclude we recall that, according 
to Proposition~\ref{pro:pts}, the subset
$\mathfrak{pts}_0\subseteq F_0$ of 
points cut out by the $(-1)$-curves of
$S$ is acted freely and transitively 
by $K_S^\perp/\Lambda\simeq
\mathbb Z/3\mathbb Z\oplus
\mathbb Z/6\mathbb Z$.
Of these $18$ points we already 
know the nine base points 
$p_1,\dots,p_9$ cut out by 
the above disjoint exceptional
curves of $S$. Moreover 
$\res(H) = 3\res(E) + \res(F_0)$,
because $\res(F_0)$ has order two,
so that $x_i \sim p_i+\res(F_0)$
is an inflection point of the plane
cubic $F_0$ in the model defined
by $H$.
This is exactly how the Chilean  configuration 
constructed in Section~\ref{chileconf}.
\end{proof}

We will give another proof of the uniqueness of the Chilean pencil in Proposition \ref{prop:unique1}.

\begin{remark}
The set of nine $(-1)$-curves 
constructed in Theorem~\ref{unique}
is invariant for any automorphism 
of the surface. Indeed let $E$ be the
$(-1)$-curve in the proof of the 
theorem. Any automorphism maps $E$ to
a $(-1)$-curve $E'$ which intersects 
any singular fiber at two curves.
Then, by the same argument in the proof
of the theorem the difference $E'-E$
must be linearly equivalent to one of the divisors $D_{0222},D_{0111},
D_{2021},D_{2210},D_{2102},D_{1012},
D_{1201},D_{1120}$.
\end{remark}

\begin{remark}
It follows from the Borel-de Siebenthal-Dynkin algorithm for embedding of root lattices that, in our case, the embedding  $\Lambda_0 := \Lambda/\langle F_0\rangle \hookrightarrow K_S^\perp/\la K_S\ra \cong {\rm E}_8$ is  unique modulo the action of 
the Weyl group $W({\rm E}_8)$.
By~\cite[Lemma 2.5]{LT} the
embeddings of $\Lambda$ into 
$K_S^\perp\simeq\tilde {\rm E}_8$ 
which are compatible
with the embedding 
$\Lambda_0\to{\rm E}_8$
are in bijection with the elements of the
group ${\rm Ext}^1_{\mathbb Z}
({\rm E}_8/\Lambda_0,\mathbb Z/2\mathbb Z)
\simeq {\rm Ext}^1_{\mathbb Z}
((\mathbb Z/3\mathbb Z)^2,\mathbb Z/2\mathbb Z)$,
which is trivial. This provides another
proof of the fact that the classes of 
$D_{0111}$ and $D_{1012}$ are in $\Pic(S)$.
\end{remark}


The base points of the Halphen pencil in \ref{eqpencil} define nine disjoint $(-1)$-curves on $S$. However, there are more $(-1)$-curves on Halphen surfaces of higher index.  

\begin{proposition}\label{minus1} The number of $(-1)$-curves on the rational elliptic surface corresponding to the Halphen pencil of Hesse type of index $2$ in \ref{eqpencil} is equal to $144$. The set of $(-1)$-curves consists of the following subsets:

\begin{itemize} 
\item $9$ exceptional curves,
\item  $36$  proper transforms of  lines through two base points, 
\item $54$ proper transforms of conics through five base points, 
\item $36$ proper transforms of cubics with a double point at one of the base points and six simple base points, 
\item proper transforms of  $9$ quartics with one triple point and passing through other 8 points.  
\end{itemize} 

The set of these $2$-sections is freely acted by the Mordell-Weil group of the jacobian fibrations isomorphic to $(\bbZ/3\bbZ)^{\oplus 2}$ with 16 orbits represented by the columns of the following matrices:

\[
 \begin{pmatrix*}[r]
0 & 2 & 2 & 2 & 2 & 2 & 2 & 4\\
0 & -1 & -1 & -1 & 0 & 0 & 0 & -1\\
0 & -1 & 0 & 0 & -1 & -1 & 0 & -1\\
0 & 0 & -1 & 0 & -1 & 0 & -1 & -1\\
0 & -1 & 0 & 0 & -1 & -1 & 0 & -1\\
0 & -1 & -1 & -1 & 0 & 0 & 0 & -1\\
0 & 0 & -1 & 0 & -1 & 0 & -1 & -1\\
0 & 0 & 0 & -1 & 0 & -1 & -1 & -1\\
0 & 0 & 0 & -1 & 0 & -1 & -1 & -1\\
1 & -1 & -1 & -1 & -1 & -1 & -1 & -3\\
\end{pmatrix*}
\qquad
 \begin{pmatrix*}[r]
1 & 1 & 1 & 1 & 3 & 3 & 3 & 3\\
-1 & 0 & 0 & 0 & -1 & -1 & -1 & 0\\
0 & -1 & 0 & 0 & -1 & -1 & 0 & -1\\
0 & 0 & -1 & 0 & -1 & 0 & -1 & -1\\
0 & -1 & 0 & 0 & -1 & -1 & 0 & -1\\
-1 & 0 & 0 & 0 & -1 & -1 & -1 & 0\\
0 & 0 & -1 & 0 & -1 & 0 & -1 & -1\\
0 & 0 & 0 & -1 & 0 & -1 & -1 & -1\\
0 & 0 & 0 & -1 & 0 & -1 & -1 & -1\\
0 & 0 & 0 & 0 & -2 & -2 & -2 & -2\\
\end{pmatrix*}
\]
The curves in the first matrix lie
restrict to the point $p_9\in F_0$ 
while the curves in the second matrix
restrict to the inflection point $x_9\in F_0$.

\end{proposition}

\begin{proof}
Recall that by Proposition~\ref{pro:bij}
the subset of $\Exc(S)$ which intersects
$F_0$ at $p$ or $p+\res(F_0)$ is in 
bijection with $L^{\rm trop}(E)/
\langle 2K_S\rangle$, where $E$ is 
any such $(-1)$-curve.
We know by Section 2 that this set is 
the set of integer points of a polytope
in $\Lambda_{\mathbb Q}/
\langle 2K_S\rangle$ which we are now
going to describe.
According to Proposition~\ref{pro:pts}
the restriction map induces a surjection
\[
 \Exc(S)\to\mathfrak{pts}_0
\]
which is equivariant with respect to the
free action of $K_S^\perp/\Lambda
+\langle F_0\rangle\simeq (\mathbb Z/3\mathbb Z)^2$.
The set $\mathfrak{pts}_0$ consists of two
orbits under this action: the set of flexes 
$\{x_1,\dots,x_9\}$ and their translates
with respect to the $2$-torsion class
$\{p_1,\dots,p_9\}$ which are cut out
by the exceptional divisors of the 
Chilean model.
Denote by $E$ one such exceptional
divisor, by $R_i^{(1)},R_i^{(2)}$ the two sides of 
the $i$-triangle which are intersected by
$E$ and by $R_i^{(0)}$ the side which is not
intersected. Then, up to multiples of
$2F_0 \sim R_i^{(0)}+R_i^{(1)}+R_i^{(2)}$, 
the only linear combination
$D := E+a_0R_i^{(0)}+a_1R_i^{(1)}+a_2R_i^{(2)}$ which
has non-negative intersection with 
all the three $(-2)$-curves is 
$E-R_i^{(0)}$.
To prove this observe that, after 
possibly adding an integer multiple 
of $2F_0$ we can assume $a_2 = 0$.
Then the three inequalities 
$D\cdot R_i^{(0)}\geq 0$, 
$D\cdot R_i^{(1)}\geq -1$, 
$D\cdot R_i^{(2)}\geq -1$
become 
$-2a_0+a_1\geq 0$,
$a_0-2a_1\geq -1$,
$a_0+a_1\geq -1$,
whose solution set is the 
triangle in the below picture.

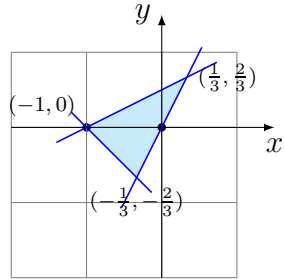
\begin{center}
\begin{tikzpicture}[scale=1]
\tkzDefPoint(0,0){O}
\tkzDefPoint(-1,0){A}
\tkzDefPoint(-1/3,-2/3){B}
\tkzDefPoint(1/3,2/3){C}
\tkzFillPolygon[color = cyan!20](A,B,C)
\tkzDrawPoints[fill=black,color=black,size=7](A,O)
\tkzInit[xmin=-2,xmax=1,ymin=-2,ymax=1]
\tkzGrid
\tkzDrawXY[noticks,>=latex]
\tkzDrawSegments[add = .3 and .3,color=blue](A,B B,C C,A)
\tkzLabelPoint[above left](A){\tiny $(-1,0)$}
\tkzLabelPoint[below](B){\tiny $(-\frac{1}{3},-\frac{2}{3})$}
\tkzLabelPoint[right](C){\tiny $(\frac{1}{3},\frac{2}{3})$}
\end{tikzpicture}
\end{center}
The only integer points in the 
triangle are $(0,0)$ and $(-1,0)$,
which correspond to the divisors
$E$ and $E-R_i^{(0)}$ respectively.
Thus a set of representatives for 
$L^{\rm trop}(E)/\langle 2F_0\rangle$
consists of all the sums $-\sum_IR_i^{(0)}$,
where $I$ is any subset of $\{1,\dots,4\}$.
As a consequence we can write down
explicitly all the curves which intersect
$F_0$ at $x_9$ or at $p_9$, these are the
following $16$ curves:
\[
 E - \sum_{i\in I}R_i^{(0)} + |I|F_0,
\]
where $I$ varies along the subsets of
$\{1,2,3,4\}$ and $|I|$ is its cardinality.
Observe that in our case we can take 
$E = E_9$, so that $E\cap F_0 = \{p_9\}$.
When $I$ has even cardinality we get the 
first set of eight curves which intersect 
$F_0$ at $p_9$.
When $I$ has odd cardinality we get the 
second set of eight curves which intersect 
$F_0$ at the flex $x_9$.
By Proposition~\ref{pro:pts} we deduce
that the number of $(-1)$-curves of $S$ 
is $8\cdot 18 = 144$.
To describe the remaining $(-1)$-curves
one can use the action of the group of
$\eta$-rational points of $\Pic^0(F_\eta)$.
This group, isomorphic to $K_S^\perp/
\Lambda+\langle 2F_0\rangle\simeq
(\mathbb Z/3\mathbb Z)^2$ acts by
translations by $3$-torsion classes 
on $F_0$. Computing the action on 
$\{p_1,\dots,p_9\}$ the nine points 
are permuted by the elements of the 
following group
\[
 \langle
 (1 9 5)(2 7 6)(3 8 4),
 (1 8 6)(2 9 4)(3 7 5)
 \rangle.
\]
The statement follows.
\end{proof}

\begin{remark}
The set $\Exc(S)$ is equipped with the
following non-geometric involution which
preserves each fiber of the map 
$\res\colon\Exc(S)\to\mathfrak{pts}_0$.
The involution is
\[
  E - \sum_{i\in I}R_i^{(0)} + |I|F_0
  \mapsto
  E - \sum_{i\in I^c}R_i^{(0)} + |I^c|F_0,
\]
where $I^c$ is the complement of $I$
in $\{1,2,3,4\}$. Observe that the two
curves have intersection number 
$-1+|I|+|I^c| = 3$ and their sum is the
divisor $2E-\sum_{i=1}^4R_i^{(0)}+4F_0$.
\end{remark}

\begin{remark}
To determine the classes of the 
$(-1)$-curves of $X$ one can proceed
as in~\cite{LT} (see also~\cite{HM}).
The classes of the $(-2)$-curves 
span a sublattice $\Lambda\subseteq K_S^\perp$ 
of rank $9$ of $\textrm{Pic}(S)$. Thus
using the perfect pairing given by the 
intersection form on the Picard group
we deduce that the polyhedra 
\[
 \Delta_S := \textrm{chull}
 (E\in\textrm{Pic}(S)\, :\, E\cdot K_S = -1
 \text{ and } E\cdot C\geq 0\text{ for any
 $(-2)$-curve $C$})
\]
has dimension $9$ and it decomposes 
as the following Minkowski sum $
 \Delta_S = \Delta + \mathbb Q\cdot K_S$,
where $\Delta$ is a polytope of dimension
$8$. It turns out that the $(-1)$-curves of
$S$ are in bijection with the integer points
of $\Delta$. Indeed clearly a $(-1)$-curve 
of $S$ has integer intersection with all
the $(-2)$-curves and thus it gives an
integer point of $\Delta$. On the other
hand an integer point of $\Delta$ can be
lifted to an affine line contained in $\Delta_S$.
This line has the form $E+nK_S$ and there
is only one integer value of $n$ such that
the self-intersection of $E+nK_S$ is $-1$.
By Riemann-Roch the class is effective
and since it has non-negative intersection 
with all the $(-2)$-curves of $S$ one concludes
that it is irreducible. 
One checks with a computer that 
$\Delta$ has $144$ integer points.
\end{remark}

In the next section we will give a proof of Proposition \ref{minus1} that does not rely on any computer computation. 

\begin{remark} Since all nine exceptional curves of $\pi:S\to \bbP^2$ form one orbit with respect to the Mordell-Weyl group of the jacobian fibration, the restriction of $f:S\to \bbP^1$ to each of them is a double cover of the base $\bbP^1$ branched over the same pair of points. The fiber of $f$ over one of the points is the  double fiber and the image under $\pi$ of the fiber over the other point is a $9$-cuspidal sextic which we discussed in Proposition \ref{dualcubic}.
\end{remark}

\section{A double plane model of $S$}

Let $S$ be an Halphen surface of index $2$ of Hesse type. We do not assume that it comes from a Chilean configuration. We denote the half-fiber by $F_0$ and the general fiber by $F$. It has also four reducible fibers $F_1,\ldots,F_4$ of type $I_3$. Fix a $(-1)$-curve $E_0$ on $S$ and let $\sigma_{E_0} \colon S\to S'$ be its blow-down morphism. The surface $S'$ is a \emph{weak del Pezzo surface} of degree 1. For example, if $E_0$ is chosen to be the exceptional curve over one of the base points of the Halphen pencil, then $S'$ is the blow-up of the remaining base points. It is well known and well documented
~\cite[\S 8.8.2]{CAG}
that the anti-canonical linear system $|-K_{S'}|$ is a pencil with one base point $s'$ and the anti-bicanonical linear system $|-2K_{S'}|$ defines a degree 2 map 
$$\phi'\colon S'\da Q\subset \bbP^3,$$
 where $Q$ is a singular quadric with vertex $v_0 = \phi'(s')$. The deck  transformation of $\phi'$ defines a biregular involution $\beta$, known as the \emph{Bertini involution} of $S'$. 

The branch curve $B'$ of $\phi'$ is cut out by a cubic surface in $\bbP^3$. If we project $Q$ to $\bbP^2$ from a point $q\not\in B'$, we obtain a \emph{double plane model} of $S'$ with branch curve a plane sextic $W$ with a triple point $q_1$ and a triple infinitely near point $q_1^{(1)}\succ q_1$. The image of $q$ is a line $L$ that intersects $W$ at $q_1$ with multiplicity $6$. However, if we project from a point $q\in B'$, then the branch curve $W$ of the double plane model is the union of a quintic curve $W_0$ with double points at $q_1,q_1^{(1)}$ and the line $L$ that intersects $W_0$ at $q_1$ with multiplicity 4 and intersects $W_0$ with multiplicity 1 at some point $q_2$.
 
We denote the composition $S \to S' \da Q \da \bbP^2$ by $\phi \colon S \da \bbP^2$. It defines a \emph{double plane model} of $S$.

Since $K_S = \sigma_{E_0}^*(K_{S'})+E_0$ and $|-K_{S}| = \{F_0\}$, we see that the point $s'$ lies on the half-fiber $F_0$. Let $s = E_0\cap F_0$. Since the birational lift of the Bertini involution $\beta$ leaves invariant $K_S$ and $\sigma_{E_0}^*(K_{S'})$, it leaves invariant $E_0$ and hence fixes the point $s$. Thus the image $q \in Q$ of $s$ belongs to $B'$. Then we project from $q$ to obtain that the branch curve $W$ of $\phi$ is the union of a plane quintic $W_0$ and the line $L$ as above.

The map $\phi$ is given by the linear system $|-2K_S+E_0|$. It blows down $E_0$ to the point $q_2$ and the image of $s'$ is the point $q_1$. It also blows down the irreducible components of the fibers $F_1,\ldots,F_4$ to singular points of $W_0$. If $E_0$ intersects two components of $F_i$, it blows down the remaining component to an ordinary node  of $W_0$. If $E_0$ intersects one component with multiplicity 2, then it blows down the remaining two components to an ordinary cusp. Thus $W_0$ has four additional double points $y_1,\ldots,y_4$. 

The map $\phi$ is a rational map of degree 2 with indeterminacy point $s$. Let $\sigma_{s,s'}\colon \tilde{S} \to S$ be the blow-up of the points $s,s'$ and $\tau \colon X \to \bbP^2$ be the blow-up of the points $q_1,q_1^{(1)},q_2,y_1,\ldots,y_4$. We have a commutative diagram
\[
\xymatrix{\tilde{S}\ar[d]^{\sigma_{s,s'}}\ar[r]^{\tilde{\phi}}&X\ar[d]^{\tau}\\
S\ar@{-->}[r]^{\phi}&\bbP^2,}
\]
where the top arrow is a finite map of degree 2. The lift $\beta_{E_0}$ of the Bertini involution generates the Galois group $G$ of the map $\tilde{\phi}$. The locus of fixed point $\tilde{S}^{\beta_{E_0}}$ consists of the union of the exceptional curves $\calE_{s'}, \calE_{s}$ and a smooth curve $B$. The map $\tilde{\phi}\circ \tau:\tilde{S}\to \bbP^2$ is given by the linear system $|F+\bar{E}_0|$, where we identify the divisor class of a general fiber $F$ of $f$ with its pull-back on $X$ and use bar to denote the proper transform.

Let $\calE_{q_1} =  \tau^{*}(q_1)$ be the exceptional configuration over $q_1$. It is equal to the sum of the exceptional curve $\calE_{q_1^{(1)}}$ over $q_1^{(1)}$ and a $(-2)$-curve $R_{q_1}$. We have 
\begin{eqnarray*}
\tilde{\phi}^*(R_{q_1}) &=& 2\calE_{s'},\\
\tilde{\phi}^*(\bar{L}) &=& 2\calE_s,\\
\tilde{\phi}^*(\calE_{q_1^{(1)}}) &=& \bar{F}_0,\\
\tilde{\phi}^*(\bar{W}_0) &=& 2 B.
\end{eqnarray*}
The elliptic pencil on $S$ is equal to the pre-image of the pencil of lines through $q_2$. The double fiber is clearly mapped to the line $L$.
The following picture  summarizes the  construction of the double plane model of $S$.

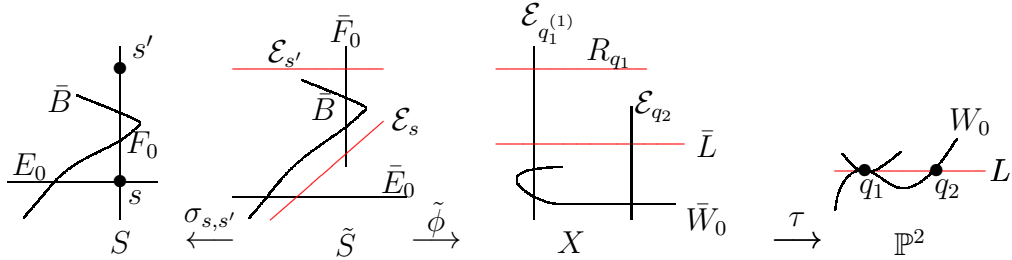
\xy (-10,25)*{};
@={(15,0),(15,15)}@@{*{\bullet}};
(0,0)*{};(20,0)*{}**\dir{-};(15,-5)*{};(15,18)*{}**\dir{-};
(18,5)*{F_0};(3,2)*{E_0};(17,-2)*{s};(18,18)*{s'};
(30,-2)*{};(53,-2)*{}**\dir{-};(45,2)*{};(45,18)*{}**\dir{-};(30,15)*{};(50,15)*{}**{\color{red}\dir{-}};
(35,-5)*{};(50,8)*{}**{\color{red}\dir{-}};
(52,0)*{\bar{E}_0};(45,20)*{\bar{F}_0};(53,8)*{\mathcal{E}_s};(37,17)*{\mathcal{E}_{s'}};
(27,-8)*{\longleftarrow};(27,-5)*{\sigma_{s,s'}};
(15,-8)*{S};(45,-8)*{\tilde{S}};
(1,-6)*+{};(8,12)*+{}
**\crv{(4,-3)&(7,2)&(15,5)&(18,8)&};
(7,11)*{\bar{B}};
(31,-6)*+{};(38,14)*+{}
**\crv{(34,-3)&(37,2)&(45,7)&(48,10)&};
(42,10)*{\bar{B}};
(70,-5)*{};(70,18)*{}**\dir{-};(65,15)*{};(85,15)*{}**{\color{red}\dir{-}};(65,5)*{};(90,5)*{}**{\color{red}\dir{-}};
(83,-5)*{};(83,10)*{}**\dir{-};
(75,2)*+{};(90,-3)*+{}
**\crv{(65,2)&(70,-3)&(75,-3)&(80,-3)&};
(75,-8)*{X};(57,-5)*{\tilde{\phi}};(57,-8)*{\longrightarrow};
(86,10)*{\mathcal{E}_{q_2}};(93,-5)*{\bar{W}_0};(80,17)*{R_{q_1}};(93,5)*{\bar{L}};
(72,21)*{\mathcal{E}_{q_1^{(1)}}};
(110,1.5)*{};(130,1.5)*{}**{\color{red}\dir{-}};
(110,-5)*+{};(127,7)*+{}
**\crv{(110,0)&(115,5)&(120,-5)};
(110,5)*+{};(120,5)*+{}
**\crv{(114,0)&};
(132,1.5)*{L};(132,1.5)*{L};(115,-1)*{q_1};(125,-1)*{q_2};(114,1.5)*{\bullet};(123.5,1.5)*{\bullet};
(120,-8)*{\mathbb{P}^2};(105,-8)*{\longrightarrow};(105,-8)*{\longrightarrow};(105,-5)*{\tau};
(128,8)*{W_0};
\endxy

\begin{remark} If we blow-up the point $s'$ on $S'$ we obtain a jacobian rational elliptic surface $\tilde{S}'$ with the section equal to the exceptional curve over $s'$. Following \cite{Lang},  the birational transformation $h:S\da \tilde{S'}$ is called  an \emph{Halphen transform}. It replaces an index 2 elliptic fibration $f:S\to \bbP^1$ with a jacobian elliptic fibration  $f':\tilde{S}'\to \bbP^1$.  Note that the elliptic fibration $f'$ is not, in general,  isomorphic to the jacobian fibration of $f:S\to \bbP^1$ although it is defined by the same data that consists of a choice of a fiber on $f'$ and a non-trivial $2$-torsion element in its Picard group (in our case defined by $\calO_{h(F_0)}(s-s')$). Note that the elliptic pencil on $\tilde{S}'$ is the pre-image of the pencil of lines through the point $q_1$.
\end{remark}

\begin{remark} In the case of a degenerate configuration where $2F_0$ is a reducible fiber, we obtain a similar  double plane model. The difference only is that the singular point $q_1$ of the quintic is not a tacnode anymore but a double cusp and the line $L$ intersects the quintic at this point with multiplicity $2$. Its proper transform on $S$ is a component of the double fiber which intersects $E_1$. The other two components are mapped to the curve $\calE_{q_1^{(1)}}$ that splits in the cover
\end{remark}  

Recall that we have proved in Lemma \ref{eight} that there exists a $(-1)$-curve $E_0$ which intersects two irreducible components of each reducible fiber of $S$. In Theorem \ref{unique} we proved the existence of $8$ $(-1)$-curves in $S$ which together with $E_0$ form a set of nine disjoint $(-1)$-curves which can be blow-down to $\mathbb P^2$ so that the images of the reducible fibers is a Chilean configuration. Let us see this from the double cover.

Let $C_i$ be a conic through $q_1,q_1^{(1)}$ and three points $y_j, j\ne i$. It does not intersect the branch curve and its  proper transform $\bar{C}_i$ on $X$ is a $(-1)$-curve. So it splits in the cover $\tilde{\phi}:\tilde{S}\to X$ into two disjoint $(-1)$-curves that are disjoint from the  exceptional divisor of $\tilde{S}\to S$. Thus, the proper transform of $C_i$ is the union of two disjoint $(-1)$-curves $E_i+E_{-i}$. Since the curves $\bar{C}_i$ are disjoint and do not pass through $q_2$, we obtain a set of 9  disjoint $(-1)$-curves $E_0,E_i,E_{-i}$ that define a birational morphism $\pi:S\to \bbP^2$.

\begin{proposition}\label{prop:unique} Every $(-1)$-curve $E_0$ can be included into a set $\calE(E_0)$ of 9 disjoint $(-2)$-curves 
$E_0,E_1,E_{-1},\ldots,E_4,E_{-4}$ such that the Bertini involution $\beta_{E_0}$ sends $E_i$ to $E_{-i}$ and leaves invariant $E_1$. Moreover, if $y_i$ is an ordinary node, the  image of the fiber $F_i$ is the union of three conics. If $y_i$ is a cusp, then the image of $F_i$ is the union of a quartic and two lines.
\end{proposition}

\begin{proof} Only the last assertion has not been proven yet. For simplicity of notation we assume that $y_i = y_1$. Suppose $y_1$ is a node. The curve $E_1$ intersects $R_1^{(1)}$ and $R_1^{(2)}$. Then each conic $C_i, i\ne 1,$ passes through $y_1$ and hence $E_{\pm i}$ intersects  $R_1^{(0)}$ and $R_1^{(1)}$ (or $R_1^{(1)}$). The curve  $E_{\pm 1}$ intersects either $R_0^{(1)}$ or $R_0^{(2)}$  with multiplicity 1 or intersects one of these curves with multiplicity 2.  In the first case each component $R_0^{(k)}$ intersects six curves from $\calE(E_0)$ and hence its image under the blowing down morphism is a conic. In the second case, $R_0^{(1)}$ intersects 4 curves $E_0,E_2,E_3,E_4$ with multiplicity 1 and one curve $E_1$ with multiplicity $2$. The self-intersection of its image under $\pi$ is equal to $-2+4+4 = 6$, a contradiction. This proves the assertion in the case when $y_i$ is a node.

Suppose $y_1$ is a cusp.   The curve $E_1$ intersects $R_1^{(0)}$ with multiplicity $2$. Each conic $C_i, i\ne 1,$ passes through $y_1$ and hence $E_{\pm i}$ intersects  both $R_1^{(1)}$ and $R_1^{(2)}$. Each curve  $E_{\pm 1}$ intersects  $R_1^{(0)}$ with multiplicity 2. Thus 
 $R_1^{(\pm 1)}$ intersects 3 curves from $\calE(E_0)$ and the curve $R_1^{(0)}$ intersects 6 curves from $\calE(E_))$ with multiplicity 1 and three curves from this set with multiplicity 2. This implies that the image of $R_1^{(\pm 1)}$ is a line and the image of $R_1^{(\pm 1)}$ has self-intersection $-2+6+12 = 16$, i.e. it is a quartic.  This proves the assertion in the case when $y_i$ is cusp.
 \end{proof}

Let $\calE(E_0)$ denote the set of 9 disjoint $(-1)$-curves including $E_0$ obtained by the construction from Proposition \ref{prop:unique}. 

\begin{definition} A set of nine disjoint $(-1)$-curves on $S$ is called a \emph{Chilean set of exceptional curves} if the image of the elliptic pencil under a birational map that blows down these curves is an Halphen pencil whose reducible members consist of three  conics.
\end{definition}

\begin{corollary} Suppose $E_0$ intersects 8 irreducible components of reducible fibers (the existence of such curve follows from Lemma \ref{eight}). Then $\calE(E_0)$ is  a Chilean set of exceptional curves and the branch curve of the double plane model does not have cusps.
\end{corollary}

\begin{proof} Since $E_0$ does not intersect any component $R_i^{(k)}$ with multiplicity 2, it follows from the proof of the previous proposition that all singular points $y_i$ are nodes.
\end{proof}

\begin{proposition}\label{prop:unique1} There is only one Chilean set of exceptional curves on $S$.
\end{proposition}

\begin{proof} Let $\calE_1$ and $\calE_2$ be two Chilean sets of exceptional curves. Let $\Sigma_i = \sum_{E\in \calE_i}E$. Since each $R_i^{(k)}$ intersects six curves from $\calE_1$ and $\calE_2$, we obtain $\Sigma_i\cdot R_i^{(k)} = 6$. Thus $\Sigma_1-\Sigma_2$ belongs to the orthogonal complement of the sublattice $\Pic_{\textrm{fib}}(S)$ of $\Pic(S)$ spanned by the components of fibers. Since $\Pic_{\textrm{fib}}(S)$ is a lattice of rank $9$ with radical spanned by $[F_0]$, its orthogonal complement is equal to $\bbZ[F_0]$. This shows that $\Sigma_1-\Sigma_2\sim nF_0$.  Then $(\Sigma_1-\Sigma_2)^2= -9 -2(9n-9)-9=0$ and so $n=0$. Therefore $\Sigma_1 \sim \Sigma_2$, but now we intersect with each $E \in \calE_1$ to get that $\Sigma_1 = \Sigma_2$.

\end{proof}

\begin{corollary}\label{invariant} The Chilean set $\calE$ of exceptional curves forms one orbit with respect to the action of $\Aut(S)$.
\end{corollary}

\begin{proof} By the uniqueness of  $\calE$, it is invariant with respect to any automorphism of $S$. The group $\Aut(S)$ contains as a subgroup $H$ the Mordell-Weil group of $j(f)$ isomorphic to $(\bbZ/3\bbZ)^2$. It acts simply transitively on the set, because it acts transitively on the set of 9 sections of $j(f)$. This can be seen by using that $H$ permutes cyclically components of each singular fiber and hence cannot fix any of the curves from $\calE$. 
\end{proof}

Let us now fix a Chilean set $\{E_1,\ldots,E_9\}$ of $(-1)$-curves. We would like to find the set $\Exc(S)$ of $(-1)$-curves on $S$ in terms the proper transforms under $\phi \colon S\da \bbP^2$ of some curves $C$ in the plane. We already know that the curves $E_2,\ldots,E_9$ come from conics passing through three of the double points of the branch quintic curve and passing through $q_1$ and its infinitely near point $q_1^{(1)}$. We fix $E_1$ that intersects $F_0$ at a nonsingular point $s$. We also have another point $s'$ such that $\calO_{F_0}(s-s') \cong {\rm res}(F_0)$. Let $\Exc_0(S)$ be the subset of $\Exc(S)$ of $(-1)$-curves passing either through $s$ or $s'$. It coincides with the set $\res^{-1}(\mathfrak{pts}_0)$ defined in Section 2. By Propositions \ref{pro:res} and \ref{pro:bij}, it coincides with $L^{{\rm trop}}(E_1)/(mK_S)$. By Proposition \ref{freeaction}, the Bertini involution $\beta_{E_1}$ acts trivially on it. Thus all curves from the set $\Exc_0(S)$ are invariant with respect to $\beta_{E_1}$. In other words, their images in the plane under the map $\phi$ do not split in the cover. 

We write  $a_s = 1$ if $s\in E$ and $a_{s'} = 1$ if $s'\in E$.

Let $E$ be any $(-1)$-curve invariant with respect to $\beta_{E_1}$. Let $R_i^{(0)},R_i^{(1)},R_i^{(2)}, i = 1,2,3,4,$ be the components of singular fibers indexed as in the proof of Lemma \ref{eight}, i.e. we assume that $E_1$ does not intersect $R_i^{(0)}$. Since the group $\Pic(S)_\bbQ^{\beta_{E_1}^*}$ of rational divisor classes invariant with respect to $\beta_{E_1}^*$ is freely generated by the divisor classes of $E_1, F_0$ and $R_1^{(0)},\ldots,R_4^{(i)}$ we can write 
$$E\sim k_0E_1-\sum_{i=1}^4k_iR_i^{(0)}+kF_0$$
for some rational coefficients. Let
$$n:= E\cdot E_1.$$
Intersecting both sides with $F_0$, we find $k_0 = 1$. Intersecting both sides with $E_1$ we find $k = n+1$. Since $E$ is invariant, it either intersects both $R_i^{(1)}$ and $R_i^{(2)}$, or intersects $R_i^{(0)}$ with multiplicity 2. Intersecting with $R_i^{(0)}$, we find that $k_i = 0$ in the former case and $k_i = 1$ in the latter case. 

 Let $v_C = \sum_{i=1}^4k_i$ be the number of nodes $y_i$ lying on $C$. The equality 
$-1 = E^2 = E_1^2+2(n+1)-2\sum_{i=1}^4k_i = -1+2(n+1)-2v_C$ gives
$$n+1 = v_C.$$

The curve $C$ passes through $v_C$ nodes, passes through $q_1$ if $a_{s'} = 1$ and passes through $q_2$ with multiplicity $n-a_s$. Since the proper transform $\bar{E}$ of $E$ on $\tilde{S}$ is the pre-image of a curve on $X$, $\bar{E}^2$ must be even (and this confirms that $a_s+a_{s'} = 1$), and since blowing only one of points on $E$, we have  $\bar{E}^2 = -2$, and hence the proper transform $\bar{C}$ on $X$ has self-intersection equal to $-1$. Since $\bar{C}$ intersects the exceptional curve $\calE_s$ over $q_2$ with multiplicity $\half (n-a_s)$, this gives 
$$\deg(C)^2=v_C+a_{s'}+\frac{1}{4}(n-a_s)^2-1.$$

Now we can list all possible $C$.

\begin{itemize}
\item[(i)] $ n = 0, v_C = 1, n = 1, q_1\in C, q_2\not\in C, \deg(C) = 1$;
\item[(ii)] $n = 1, v_C = 2, a_{s} = 1,   q_1, q_2\not\in C, \deg(C) = 1$;
\item[(iii)] $n = 2,  v_C = 3, a_{s'} = 1,  q_1, q_2\in C, \deg(C) = 2$;
\item[(iv)] $n = 3,  v_C = 4, a_{s} = 1,  q_1\in C, q_2\not\in C, \deg(C) = 2$;
\end{itemize}
In another words, $C$ is either a line passing  through $q_2$ and one node (case (i)), or a line passing through two nodes (case (ii)),  or a conic passing  passing through $q_1,q_2$ and three points $y_i$ (case (iii)), or a conic passing 
through $q_2$ and all points $y_i$ (case (iv)).

Let $(e_0,e_1,\ldots,e_9)$ be the geometric basis corresponding to the base points of the Chilean Halphen pencil. We take the double model corresponding to $E_1$. Then the divisor class of 
 a $(-1)$-curve $E$ is $de_0-\sum_{i=1}^9$. 
We have $m_1 = E\ cdot E_1 = n$. Since $E_2,\ldots,E_9$ are the split pre-images of the conics through three nodes and $q_1,q_1^{(1)}$,  we can group them in pairs and redenoting them by $E_2,E_{-2},E_3,E_{-3},E_4,E_{-4}$. We assume that  $\phi(E_i+E_{-i})$ is the conic passing through all points $y_i$ except $y_i$.  Intersecting $C = \phi(E)$ with one of the curves, we find the divisor class of $[E]$ to be equal to 

$$[E] = \begin{cases}
  e_0-e_{i}-e_{-i} & \text{in case (i)},\\
  2e_0-e_1-e_i-e_{-i}-e_j-e_{-j}&\text{in case (ii)}, \\
  3e_0-2e_1-\sum_{j\ne i}(e_j+e_{-j})&\text{in case (iii)}, \\
     4e_0-3e_1-\sum (e_i+e_{-i}) & \text{in case (iv)}.
\end{cases}
$$

As we see that, together with $E_1$  there are 16 such curves as expected from our general theory.

 Since each $E_i\cap \beta_{E_1}(E_i) = \emptyset$ for $i\ne 1$, the branch curve $B$ does not intersect such $E_i$. It intersects $E_1$ with multiplicity $1$. This gives 
\beq\label{branchclass}
[B] = e_0-e_1
\eeq
 In other words, the image of $B$ in the plane is a line passing through $x_1$.

To find the rest of $144-16$ curves, we use that the translation group acts simply transitively on the set of curves $E_1,E_i,E_{-i}$. Thus it leaves the class $\sum_{i=1}^9e_i$ invariant, and hence leaves invariant the class $e_0$ that defines the blowing down $\pi:S\to \bbP^2$. This means that the group of translations descend to a group of projective transformations of the plane. Using the explicit formulas for the base points of the Chilean pencil from section 5, we find that the group is generated by transformations $(x:y:z)\mapsto (z:y:x)$ and $(x:y;z)\mapsto (x:\epsilon y:\epsilon^2 z)$. Now the rest of the curves are obtained by applying these transformations to the 16 that has been already found. The following Table gives the result, and it also gives the images of $(-1)$-curves under $\phi$.

\begin{table}[!htbp]
  \centering 
  \scalebox{0.9}{
  $\displaystyle  
\begin{array}{|c|c|c|c|c|c|c|c|c|c|}
\hline
\deg(C)&n&v_C&u_C&a_s&a_{s'}&\deg(\pi(E))&{\rm split}&{\rm class}&\#\\ \hline
0&0&0&0&0&0&0&{\rm no}&e_1&1\\
1&0&1&0&0&1&1&{\rm no}&e_0-e_{i}-e_{-i}&4\\
1&1&2&0&1&0&2&{\rm no}&2e_0-e_1-e_i-e_{-i}-e_j-e_{-j}&6\\
2&2&3&0&0&1&3&{\rm no}&3e_0-2e_1-\sum_{j\ne i}(e_j+e_{-j})&4\\
2&3&4&0&1&0&4&{\rm no}&4e_0-3e_1-\sum (e_i+e_{-i})&1\\ \hline
2&0&3&0&0&0&0&{\rm yes}&e_{\pm i}&8\\
2&0&2&0&0&0&1&{\rm yes}&e_0-e_{\pm i}-e_{\pm j}&24\\
2&0&1&0&0&0&2&{\rm yes}&2e_0-e_i-e_{-i}-e_{\pm j}-e_{\pm k}-e_{\mp l}&24\\
2&0&0&0&0&0&3&{\rm yes}&3e_0-2e_{\pm i}-\sum_{j\ne i}(e_j+e_{-j})&8\\ \hline
3&1&3&1&0&0&1&{\rm yes}&e_0-e_1-e_{\pm i}&8\\
3&1&2&1&0&0&2&{\rm yes}&2e_0-e_1-e_i-e_{-i}-e_{\pm j}-e_{\mp k}&24\\
3&1&1&1&0&0&3&{\rm yes}&3e_0-e_1-e_i-e_{-i}-2e_{\pm j}-e_{\mp j}-e_{\pm k}-e_{\pm l}&24\\
3&1&0&1&0&0&4&{\rm yes}&4e_0-e_1-3e_{\pm i}-e_{\mp i}-\sum_{j\ne i}(e_j+e_{-j})&8\\ \hline
\end{array}$}
\vspace{0.3cm}
\caption{The classes of the $144$ $(-1)$-curves on $S$}
\label{144}
\end{table}
Here $u_C$ denotes the number of double points of $C$ among $y_1,\ldots,y_4$. We see that  all 144 exceptional curves are accounted for.

\begin{remark}\label{ptF}
An explicit computer computation gives the equation of the branch quintic curve $W_0$: 
$$
\begin{array}{l}
 x^3y^2 
 + 2\epsilon x^2y^3 
 + \epsilon^2xy^4 
 + 2\epsilon^2x^3yz 
 + (2a^3 + 4)x^2y^2z\\ 
 + 2\epsilon a^3xy^3z 
 + 2\epsilon^2(2a^3 - 1)y^4z 
 + \epsilon(-4a^3 + 1)x^3z^2 
 + \epsilon^2(-10a^3 +4)x^2yz^2 
 + (a^6 - 12a^3 - 4)xy^2z^2\\ 
 + 4\epsilon(a^3 - 2)y^3z^2 
 + (-4\epsilon a^6 - 16ea^3 + 2e)x^2z^3 
 - 8\epsilon^2(5a^3 + 1)xyz^3 
 + (2a^6 - 32a^3 - 16)y^2z^3 \\
 + (-16\epsilon a^6 - 16ea^3 - 4e)xz^4 
 - 8\epsilon^2(5a^3 +2)yz^4 
 + (-16\epsilon a^6 - 8\epsilon)z^5 = 0
 \end{array}
$$
\end{remark}

Now consider the blow-up of the base points of the Halphen pencil defining a Chilean configuration. Fix one of the exceptional curves $E_i$, and consider the double plane model $\phi_i \colon S \dashrightarrow \bbP^2$ defined by $E_i$. Using \eqref{branchclass}, we find that the branch curve $B_i$ of $\phi_i$ is represented in the geometric basis by $e_0-e_i$. So the image of $B_i$ under the blowing down map $\pi \colon S\to \bbP^2$ is a line $\ell_i$ passing through the base point $x_i$. Since $B_i$ intersects each reducible fiber at one of its singular point and intersects the component opposite to this point with multiplicity 2, we see that $\ell_i$ passes through the  intersection points of four conics in the Chilean configuration that contain the base point $x_i$. This defines a configuration of 9 lines and 12 points isomorphic to the dual Hesse configuration $(9_4,12_3)$.
 
We have proved the following.
 
\begin{theorem} A Chilean configuration of 12 conics defines nine lines $\ell_i$ and  12 intersection points of conics isomorphic to the dual Hesse configuration $(9_4,12_3)$.   
\end{theorem} 
 
\begin{remark} 
\label{rem:inv}
Note that the Hesse pencil is defined over $\bbF_4$ if we assume that the characteristic is equal to 2. Its set of 12 line components of fibers, 9 harmonic polars and the set of 9 base points and 12 singular points of fibers is equal to the set of 21 lines and 21 points in $\bbP^2(\bbF_4)$. The specific of characteristic 2 is that each harmonic polar line passes through the corresponding base point. This is similar to what we have in the case of the Chilean pencil. The union of 21 lines and 21 points form a symmetric configuration $(21_5)$. Let $X$ be the blow-up of all 21 points. It admits an inseparable finite map of degree 2 isomorphic to a supersingular K3 surface with Artin invariant 1 \cite{DolgKondo}. It is a minimal resolution of the double plane
$$w^2+xyz(x^3+y^3+z^3) = 0.$$

The Chilean pencil is defined over any non-trivial extension $K = \bbF_4(a)$ of the finite field $\bbF_4$ and defines a configuration of $21$ points, $12$ conics and $9$ lines over $K$. There is no degenerate configuration in this case because the multiplicative group of $\bar{K}$ does not have elements of order $2$. We again blow-up the 21 points, and consider the double plane $$w^2+F(x,y,z) = 0,$$
where $F(x,y,z)$ is any smooth irreducible non-multiple member of the pencil.
The double cover has $21$ ordinary double points defined over $K$, and its minimal resolution is a supersingular K3 surface.  Our 1-dimensional family of such surfaces is one of the three irreducible components of the moduli space of supersingular K3 surfaces with Artin invariant $\sigma$ equal to 2 studied  by I. Shimada \cite{Shimada}. Each family is isomorphic to the affine line with one point deleted. The closures of these families contain one common point that corresponds to the supersingular K3 surface with the Artin invariant  $1$. The three families are distinguished by a certain $9$-dimensional linear code in $\bbF_2^{21}$. Our family has the weight polynomial
$$W(t) = 1+9t^5+102t^8+144t^9+144t^{12}+102t^{13}+9z^{16}+z^{21}.$$
The code is generated by 9 words of weight $5$ representing our harmonic polar lines $B_i$ and $66$ words of weight $8$ represented by our $54$ conics and $12$ conic components of reducible fibers. 
\end{remark}

\begin{remark} An inspection of the list $\Exc(S)$ of 144 $(-1)$-curves from Proposition \ref{minus1} suggests a certain involution on this set that interchanges the exceptional curves $E_i$ with quartics, lines with cubics and pairs conics. An explanation of this duality is the following. Given a $(-1)$-curve $E$ we consider  the corresponding double plane and its ramification curve $B$ on $S$. We have $E\cdot B = 1$, so that $(F+B-E)^2 = -1$ and $K_S\cdot (F+B-E) = -1$.  It is easy to see that, for any $(-2)$-curve $R$, we have either $R\cdot (F+B-E)\ge 0$ (if $R$ intersects $E$, then it also intersects $B$). Applying Lemma \ref{lem:-1}, we see that $F+B-E \sim E'$, where $E'$ is a $(-1)$-curve. This defines an involution on the set $\Exc(S)$. The inspection of Table \ref{144} confirms that it matches the curves as above.
\end{remark}

\begin{remark} Any five points determine a conic, but we want to exclude the $12$ conics which are part of reducible fibers.  Such forbidden sets of fives corresponds to removing one point of the six of such on each component conic, thus the number of allowed sets of fives are given by ${9 \choose 5}-6\cdot12=54$. Examples of such sets are those given by the union of two non-parallel lines in our affine space of nine points. The number of such sets can easily be determined to be $54$ ($12\cdot(12-3)/2$ or $9\cdot{4\choose2}$) thus all our allowable sets are of this type. This makes the duality $I\to I'$ explicit. Given $I$ with special point $p$ (the intersection of the lines) then choose $I'$ as the two residual line through $p$ (there are $4$ lines through each point).   
\end{remark}

\begin{remark}
It is known that two $(-1)$-curves on a weak del Pezzo surface of degree $d$ intersect with multiplicity $\le 3$ ($2$ if $d = 2$ and $1$ if $d\ge 3$ because in the latter case they become lines in the anti-canonical model). An Halphen surface is obtained by blowing up one point from a weak del Pezzo surface of degree 1. We observed that in our case two $(-1)$-curves intersect with multiplicity $\le 3$. This raises the question:\textit{What is the maximal intersection number of two $(-1)$-curves on an Halphen surface with extremal jacobian fibration?}  
\end{remark}

\section{Higher indices}

Let $m$ be a positive integer which is 
not multiple of $3$ and let $S_m$ be
an elliptic Halphen surface of index $m$ whose
jacobian is the Hesse surface.
Such surface can be constructed as in
the case $m=2$ by blowing-up the plane
at the nine points $p_1,\dots,p_9$, where
$p_i := x_i + \eta$ and $\eta$ is an 
$m$-torsion point of the curve of equation
$x^3+y^3+z^3+txyz = 0$.
We say that the points $p_i,p_j,p_k$ are 
{\em Hesse-collinear} if $x_i,x_j,x_k$ 
are collinear. We consider two cases.
\begin{enumerate}
\item
If $m = 3k+1$, then there are $12$ irreducible
plane curves, each of degree $m$ with
three Hesse-collinear points of multiplicity $k+1$ and 
the remaining six points of multiplicity $k$.\\
\item
If $m = 3k+2$, then there are $12$ irreducible
plane curves, each of degree $m$ with
three Hesse-collinear points of multiplicity $k$ and 
the remaining six points of multiplicity $k+1$.
\end{enumerate}
In both cases one can verify the existence
of the curves by checking that the points with 
the given multiplicities sum $0$ in the group
law of the cubic. The irreducibility is a consequence 
of the fact that a curve of self-intersection $-2$ on 
a relatively minimal elliptic surface can only 
be sum of irreducible $(-2)$-curves, but we 
cannot have more than $12$ of them.

Here are some $m$-torsion loci with respect to
the point $x_7 = (1:-1:0)$.
\footnote{We have computed these loci with
Magma moving to a Weierstra\ss{} model of the
elliptic curve and calculating the $m$-torsion
subscheme there.}
When $m = 4$ the locus is cut out by 
\[
 xy^3 - y^4 - txy^2z - xz^3 - 2yz^3 = 0.
\]
When $m = 5$ the locus is cut out by 
\begin{gather*}
  2x^2y^6 - xy^7 + 2y^8 - x^2y^3z^3 - xy^4z^3 + 5y^5z^3
 - x^2z^6 + 2xyz^6 + 2y^2z^6\\
 + t
 (-x^2y^5z + 3xy^6z - y^7z + x^2y^2z^4 + 3xy^3z^4 + yz^7)\\
 + t^2
 (x^2y^4z^2 - xy^5z^2  + xy^2z^5)  = 0
\end{gather*}

To construct elliptic surfaces of index 3
whose jacobian is the Hesse surface we 
proceed as follows.
Let $\pi\colon 
\mathbb Z/9\mathbb Z\oplus\mathbb Z/3\mathbb Z
\to \mathbb Z/3\mathbb Z\oplus\mathbb Z/3\mathbb Z$
be the homomorphism induced by the projection
on the first coordinate. Let $\beta_{E_0}$ be a 
set-theoretic section of $\pi$ such that
the sum of all the elements in its image 
generates the kernel of $\pi$. An example
of such a $\beta_{E_0}$ is given by the columns
of the following matrix
\[
 \begin{pmatrix*}[r]
 0&0&0&1&4&1&5&5&5\\
 0&1&2&0&1&2&0&1&2
 \end{pmatrix*},
\]
where the first coordinate of each vector
is in $\mathbb Z/9\mathbb Z$ while the
second is in $\mathbb Z/3\mathbb Z$.
Choose nine points $p_1,\dots,p_9$
on a smooth plane cubic $C$ in such a way
that $p_1$ is a flex of $C$ and the class
of $p_i-p_1$ is given by the $i$-th column
of the above matrix. The blow-up of 
$\mathbb P^2$ at the nine points is an
elliptic surface of index $3$ whose jacobian
is the Hesse surface and whose $(-2)$-classes 
are the columns of the following matrix
\[
 \begin{pmatrix*}[r]
1&4&4&1&4&4&1&4&4&1&4&4&\\
-1&-1&-1&-1&-1&-1&0&-2&-1&0&-2&-1&\\
-1&-1&-1&0&-2&-1&-1&-1&-1&0&-1&-2&\\
-1&-1&-1&0&-1&-2&0&-1&-2&-1&-1&-1&\\
0&-2&-1&0&-1&-2&0&-2&-1&0&-1&-2&\\
0&-2&-1&-1&-1&-1&-1&-1&-1&-1&-1&-1&\\
0&-2&-1&0&-2&-1&0&-1&-2&0&-2&-1&\\
0&-1&-2&0&-2&-1&0&-2&-1&-1&-1&-1&\\
0&-1&-2&0&-1&-2&-1&-1&-1&0&-2&-1&\\
0&-1&-2&-1&-1&-1&0&-1&-2&0&-1&-2&\\
 \end{pmatrix*}.
\]
The $9$-torsion points with respect to
the point $x_7 = (1:-1:0)$ are cut out 
by the following cubics:
\[
 \begin{array}{ll}
    xy^2 + \epsilon x^2z + \epsilon^2yz^2, &
    xy^2 + \epsilon^2x^2z + \epsilon yz^2, \\
    xy^2 + x^2z + yz^2, &
    x^2y + \epsilon^2y^2z + \epsilon xz^2,\\
    x^2y + \epsilon y^2z + \epsilon^2xz^2, &
    x^2y + y^2z + xz^2,\\
    3x^3 + (\epsilon  + 2)txyz  - 3\epsilon^2z^3, &
    3x^3 + (-\epsilon  + 1)txyz - 3\epsilon z^3
 \end{array}
\]

\section{Applications} 

In this section, we exhibit some nice properties of the Chilean configuration. 

A simple crossings configuration of curves on a surface $Z$ is a collection of nonsingular irreducible curves $\sfA =\{ C_1,\ldots,C_d \}$ such that $C_i, C_j$ are transversal for all $i \neq j$. For $n \geq 2$, an n-point of $\sfA$ is a point which belongs to exactly $n$ curves in $\sfA$. Let $\sigma : Y \to Z$ be the blow-up of all n-points with $n\geq 3$. Then $D:= \sigma^*(\sfA)_{\text{red}}$ is a simple normal crossings divisor, and we define the log Chern numbers $\bar{c}_1^2, \bar{c}_2$ of $(Z,\sfA)$ as the Chern numbers of $\Omega_{Y}^1\log(D)$. We have (see \cite[\S4]{Urz}) $$\bar{c}_1^2= c_1^2(Z) - \sum_{i=1}^d R_i^{(2)} + \sum_{n \geq 2} (3n-4)t_n + 4 \sum_{i=1}^d (g(C_i)-1) $$ $$\bar{c}_2= c_2(Z) + \sum_{n \geq 2} (n-1)t_n + 2 \sum_{i=1}^d (g(C_i)-1),$$ where $t_n$ is the number of n-points. The highest the log Chern slope $\bar{c}_1^2 / \bar{c}_2$ is, the more special the configuration is. As the Chilean configuration of 12 conics is simple crossings with $t_2=12$, $t_8=9$, and $t_n=0$ else, we have $\bar{c}_1^2=117$, $\bar{c}_2=54$, and so  
$\bar{c}_1^2 / \bar{c}_2 = 13/6 = 2.1\bar{6}$. The Hesse configuration of 12 lines has $\bar{c}_1^2 / \bar{c}_2 =2.5$.

We recall that log Chern slopes do have constraints. As a general example, let us consider a simple crossings arrangement of plane curves $\{ C_1,\ldots,C_d \}$ with $\bigcap_{i=1}^d C_i= \emptyset$. It is a divisible arrangement as defined in \cite[\S4]{Urz} (see \cite[Example 4.3]{Urz}), and so by \cite[Theorem 6.1]{Urz}, if $\bar{c}_2 \neq 0$, we can compare $\bar{c}_1^2 / \bar{c}_2$ with the Chern slope of a smooth complex projective surface. Typically these surfaces are not ruled, and so we can use the Bogomolov-Miyaoka-Yau inequality for complex algebraic surfaces to show that $\bar{c}_1^2 \leq 3 \bar{c}_2$. See \cite{EFU} for a systematic study of log Chern slopes for configurations of lines. In that case, we have the combinatorial (independent of the field of definition) inequalities $\bar{c}_2 \leq \bar{c}_1^2 \leq 3 \bar{c}_2$ (up to trivial configurations). For complex configurations of lines, the inequalities improve to the sharp $\bar{c}_2 \leq \bar{c}_1^2 \leq 8/3 \bar{c}_2$, where the upper bound is achieved if and only if the configuration is the dual Hesse configuration. It is easy to prove (combinatorially) that $[1,2[$ has no limit points for log Chern slopes over any fixed field. In \cite{EFU} it is proved that log Chern slopes are dense in $[2,2.5]$ over $\mathbb C$, and it is conjectured that there are no limit points in $]2.5,2.\bar{6}]$. In positive characteristic, it is proved that they are dense in $[2,3]$.      

There are at least four special configurations after we add the degenerated Chilean configuration, and the dual Hesse configuration into the picture:

\textbf{($\sfA_0$)} The degenerated Chilean configuration is a simple crossings configuration of 9 conics and 3 lines. We have $t_2=12$ and $t_7=9$, and so $\bar{c}_1^2=99$, $\bar{c}_2=45$, and $\bar{c}_1^2/\bar{c}_2=11/5=2.2$. Thus when we go from the Chilean configuration to the degenerated Chilean configuration, the log Chern slope jumps up. 

\textbf{($\sfA_1$)} Consider the configuration $\sfA_1$ consisting of the $12$ irreducible conics of the Chilean configuration, and the $9$ lines of the dual Hesse configuration. Then $\sfA_1$ has $t_2=6 \times 3 \times 4=72$, $t_5=12$, $t_8=9$, $t_n=0$ else. Therefore $\bar{c}_1^2=324$, $\bar{c}_2=144$, and $\bar{c}_1^2/\bar{c}_2=9/4=2.25$.

\textbf{($\sfA_2$)} Consider instead the configuration $\sfA_2$ formed by one of the four degenerated Chilean configurations of 9 conics and 3 lines together with the 9 lines in the dual Hesse configuration. Then $\sfA_2$ has $t_2=6 \times 3 \times 3=54$, $t_5=12$, $t_7=9$, $t_n=0$ else. Therefore $\bar{c}_1^2=270$, $\bar{c}_2=117$, and $\bar{c}_1^2/\bar{c}_2=30/13=2.\overline{307692}$.

\textbf{($\sfA_3$)} This is just putting together the Hesse configuration of 12 lines with the dual Hesse configuration of 9 lines as appear in the Chilean configurations. Then we have $t_2= 9 \times 4=36$, $t_4=9$, $t_5=12$, $t_n=0$ else. A nice observation is that the  $9 \times 4$ 2-points are the $9 \times 4$ base points of the degenerated Chilean arrangements. We have $\bar{c}_1^2=180$, $\bar{c}_2=72$, and $\bar{c}_1^2/\bar{c}_2=5/2=2.5$.
\label{AllinAll}

Another application is suggested by Anatoly Libgober. Let $\calC$ be the union of 12 conics from the Chilean configuration. 

\begin{proposition} The fundamental group $\pi_1(U)$ of the complement $U = \bbP^2\setminus \calC$ fits in the exact sequence
$$F_{19}\to \pi_1(U)\to F_3\to 1,$$
where $F_k$ denote the free group with $k$ generators.
\end{proposition}

\begin{proof} The restriction of the Halphen pencil to $U$ defines a morphism $U\to \bbP^1\setminus \{4 \, \textrm{points}\}$. Its fibers are isomorphic to plane sextic curves with 9 nodes deleted. Passing to the normalization we obtain that the fiber is isomorphic to an elliptic curve with 18 points deleted. Its fundamental group is isomorphic to the free group $F_{19}$. The exact sequence from the assertion is the exact sequence for fundamental groups of smooth fibrations. 
\end{proof}

We refer to \cite{Libgober} for the discussion of the problem of classification of configuration of curves on a simply connected algebraic surface that admits a surjection to a free non-commutative group.

Another application of our configuration can be found in \cite{Pokora}. Recall that a plane curve $C$ is called \emph{free} if the logarithmic tangent sheaf $\Theta^1(\log C)$ is locally free, or, equivalently, the Jacobian sheaf of $C$ has projective dimension 2.

It is well-known that the union of 12 lines from the Hesse configuration, being a special case of an arrangement of reflection hyperplanes of a complex reflection group, is a free divisor.

The following proposition is proved in \cite[Proposition 1 and 2]{Pokora};

\begin{proposition} The union of 12 conics from the Chilean configuration as well as the union of 12 conics and 9 lines from the embedded dual Hesse configuration are free divisors in the projective plane.
\end{proposition}


Another interesting attribute of a reducible curve $C$ on a rational surface is the \emph{Harbourne constant} $H(C)$ defined by 
$$H(C) = \frac{C^2-\sum_{x\in X}\mult_x(C)^2}{s},$$
where $\mult_x(C)$ is the multiplicity of a singular point $x\in C$ and $s$ is the number of such points. It is conjectured that $H(C)\geq -4.5$ for curves over $\mathbb C$.

In our case we may consider the union $C$ of nine 2-sections $E_i$ and four reducible fibers of the elliptic fibration on $S$. We have $C^2 = -9+9.16= 15.9 = 135$ and we have $84 =8.9+12$ points of multiplicity 2, so $H(C) = \sum_{x\in X}\mult_x(C)^2= -3+\frac{17}{28}$. On the other hand, if we take the degenerate Chilean configuration, we obtain $C^2 = -9+14.9 = 117$ and $\sum_{x\in X}\mult_x(C)^2 = 7.9+12 = 75$  that gives $H(C) = -3+\frac{14}{25}$, that is slightly less. 

We also point out that this configuration of $12$ conics has also recently appeared in \cite{KRS} in the context of generalized Kummer surfaces. We refer to that pre-print for details.


\end{document}